\documentclass[12pt]{article}

\usepackage{amssymb}
\usepackage{amsthm}
\usepackage{amsmath}
\usepackage{graphicx}
\usepackage{fullpage}
\usepackage{subcaption}
\usepackage{color}
\usepackage{enumerate}
 \numberwithin{equation}{section}
\usepackage{booktabs}
\allowdisplaybreaks

\usepackage[pdftex]{hyperref}
 \usepackage{mathtools}
 \mathtoolsset{showonlyrefs}
 \usepackage[nocompress]{cite}

\theoremstyle{plain}
\newtheorem{thm}{Theorem}[section]
\newtheorem{cor}[thm]{Corollary}

\newtheorem{lem}[thm]{Lemma}
\newtheorem{prop}[thm]{Proposition}

\theoremstyle{definition}

\newtheorem{ex}[thm]{Example}

\theoremstyle{remark}

\newtheorem{rem}[thm]{Remark}

\newcommand{\N}{\mathbb{N}}
\newcommand{\R}{\mathbb{R}}

\newcommand{\es}{\mathbb{S}}

\newcommand{\bp}{\begin{proof}[\ensuremath{\mathbf{Proof}}]}
\newcommand{\bs}{\begin{proof}[\ensuremath{\mathbf{Solution}}]}
\newcommand{\ep}{\end{proof}}
\newcommand{\be}{\begin{equation}}
\newcommand{\ee}{\end{equation}}

\begin{document}

\title{The Blaschke--Lebesgue theorem revisited}

\author{Ryan Hynd\footnote{Department of Mathematics, University of Pennsylvania. This work was supported in part by an American Mathematical Society Claytor--Gilmer fellowship. }}
\maketitle

\section{Introduction}
A convex and compact subset of the Euclidean plane is a {\it constant width shape} provided the distance between two parallel supporting lines is the same in all directions. For convenience, we will always assume this distance is equal to one.  In addition, we will say that the boundary of a constant width shape is a {\it constant width curve} and refer to constant width shapes and their boundary curves interchangeably. A simple example of a constant width curve is a circle of radius $1/2$.  Indeed, any two parallel supporting lines touch the circle at the ends of a diametric chord.  As we shall see below, there is a plethora of constant width curves.
\begin{figure}[h]
     \centering
         \includegraphics[width=.6\columnwidth]{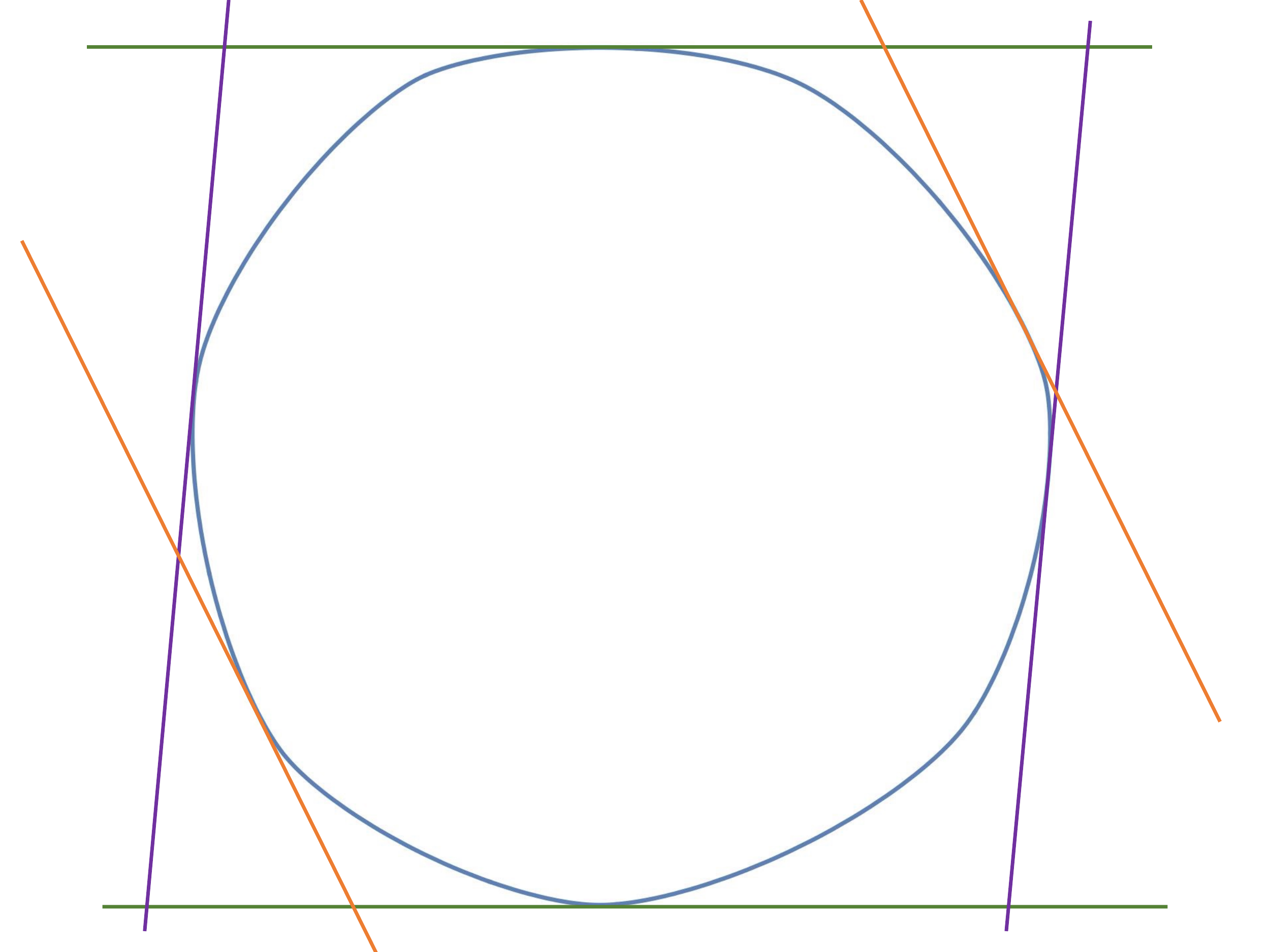}
                   \caption{A constant width curve with three pairs of parallel supporting lines.}
                   \label{FirstFig}
\end{figure}
\par  In what follows, an important example of a constant width curve is a Reuleaux triangle.  This shape is obtained as the intersection of three closed disks of radius one which are centered at the vertices of an equilateral triangle of side length one.  In order to check that this shape has constant width, we only need to make the following observation. For any pair of parallel supporting lines, one touches a vertex of the associated equilateral triangle and the other touches a point on the circle of radius one centered at this vertex. As a result, the distance between these lines is necessarily equal to one. More generally, a Reuleaux polygon is a curve of constant width consisting of finitely many arcs of circles of radius one. We note that the constant width condition necessitates that these shapes have an odd number of sides.  
\begin{figure}[h]
     \centering
     \begin{subfigure}[b]{0.37\textwidth}
         \centering
         \includegraphics[width=\textwidth]{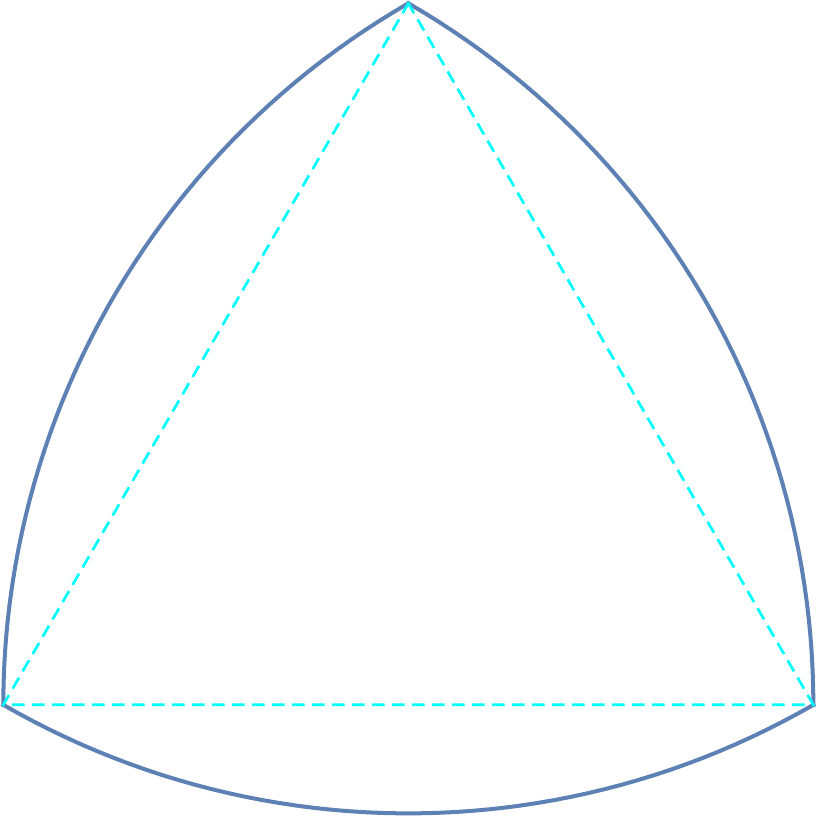}
                   \caption{A Reuleaux triangle with inscribed equilateral triangle.}
     \end{subfigure}
    \hspace{.2in}
     \begin{subfigure}[b]{0.54\textwidth}
         \centering
         \includegraphics[width=\textwidth]{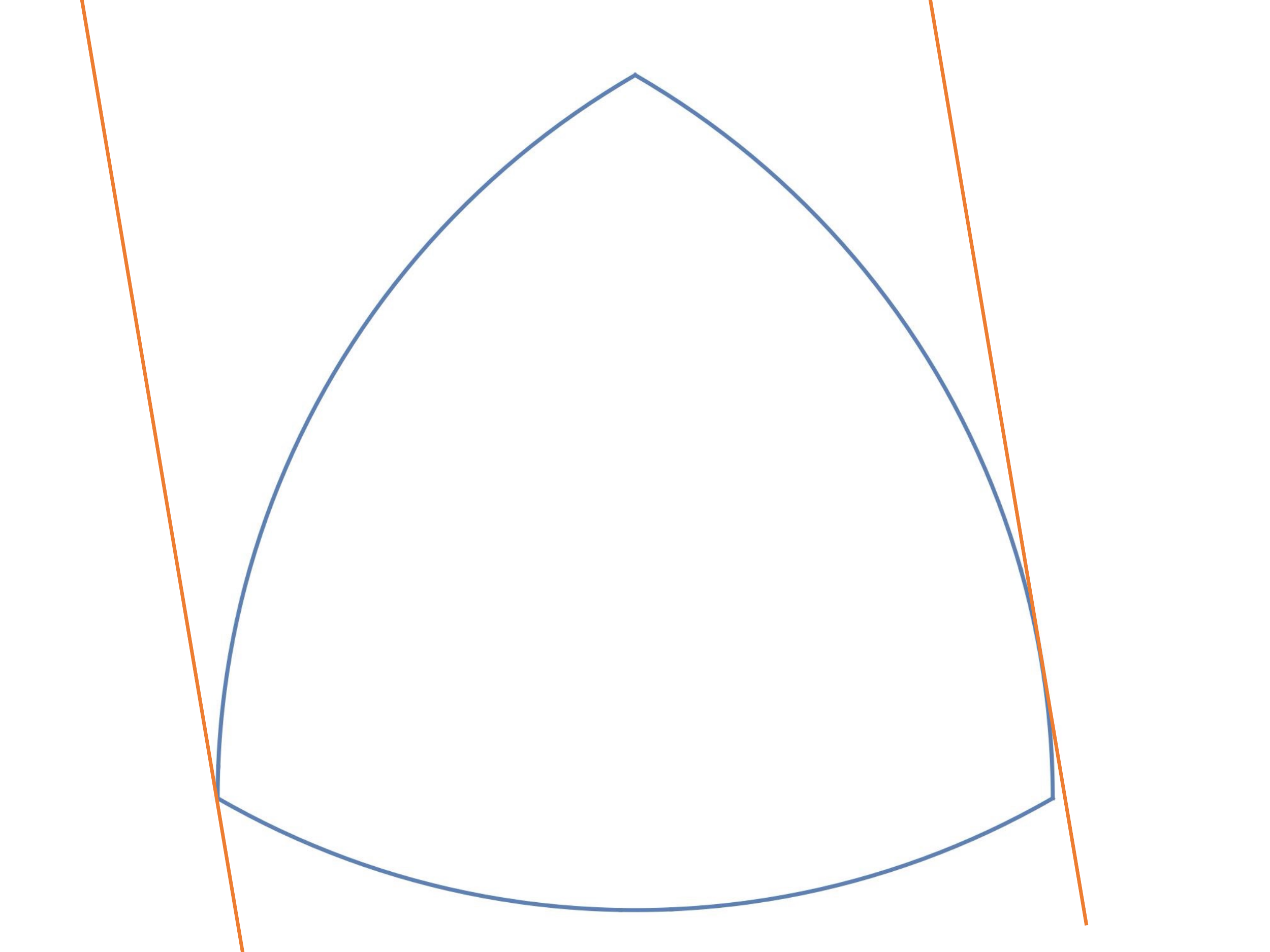}
                   \caption{A Reuleaux triangle with a pair of parallel supporting lines.}
     \end{subfigure}
\end{figure}

\par A fundamental theorem in the study of constant width curves is due independently to Blaschke \cite{MR1511839} and Lebesgue \cite{Lebesgue}.
\\\\
\noindent {\bf The Blaschke--Lebesgue theorem}. Among all curves of constant width, Reuleaux triangles enclose the least area. 
\\\\
\noindent The purpose of this note is to survey a proof of the Blaschke--Lebesgue theorem.   Our aim is to be as  self--contained as possible without going too far astray from our goal of understanding why this theorem holds.  

\par 
 This note is organized as follows.  We will first review the support function of a constant width shape; this is the principal tool employed throughout this paper.  Then we will verify the Blaschke--Lebesgue theorem using fact that constant width curves can be closely approximated by Reuleaux polygons.  This method was first used by Blaschke \cite{MR1511839}, although our argument is based on an exercise in the classic textbook of Yaglom and Boltyanski\u{\i} \cite{MR0123962}. We also acknowledge that there are several other proofs of the Blaschke--Lebesgue theorem including those given in \cite{MR51543,MR1391153,MR205152,MR1568234,MR2559951,MR1371579}. 

\begin{figure}[h]
     \centering
     \begin{subfigure}[b]{0.44\textwidth}
         \centering
         \includegraphics[width=\textwidth]{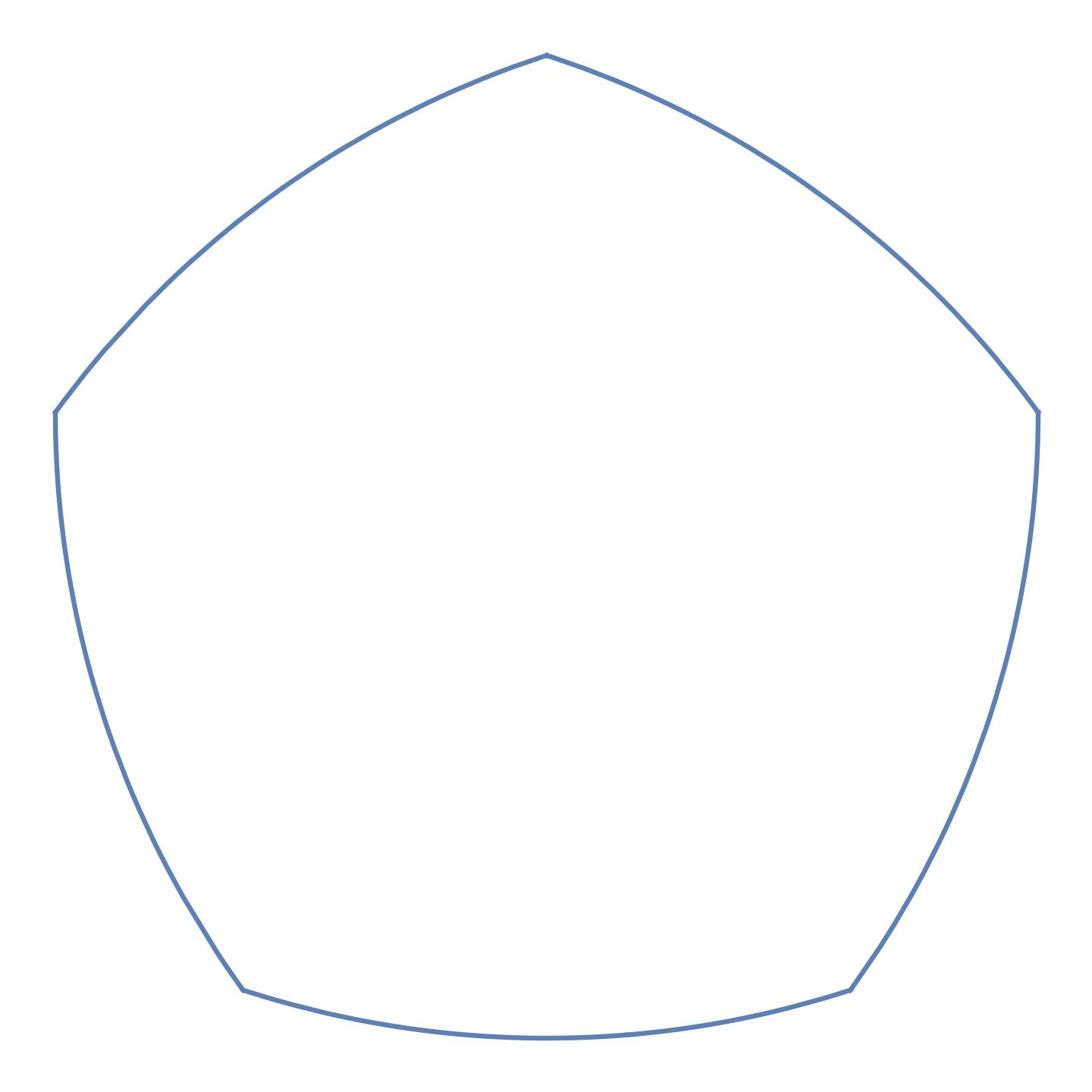}
                   \caption{A Reuleaux pentagon for which all circular arcs are the same length. We call this curve a regular Reuleaux pentagon. }
                   \label{PentFig}
     \end{subfigure}
    \hspace{.1in}
          \begin{subfigure}[b]{0.42\textwidth}
      \includegraphics[width=\columnwidth]{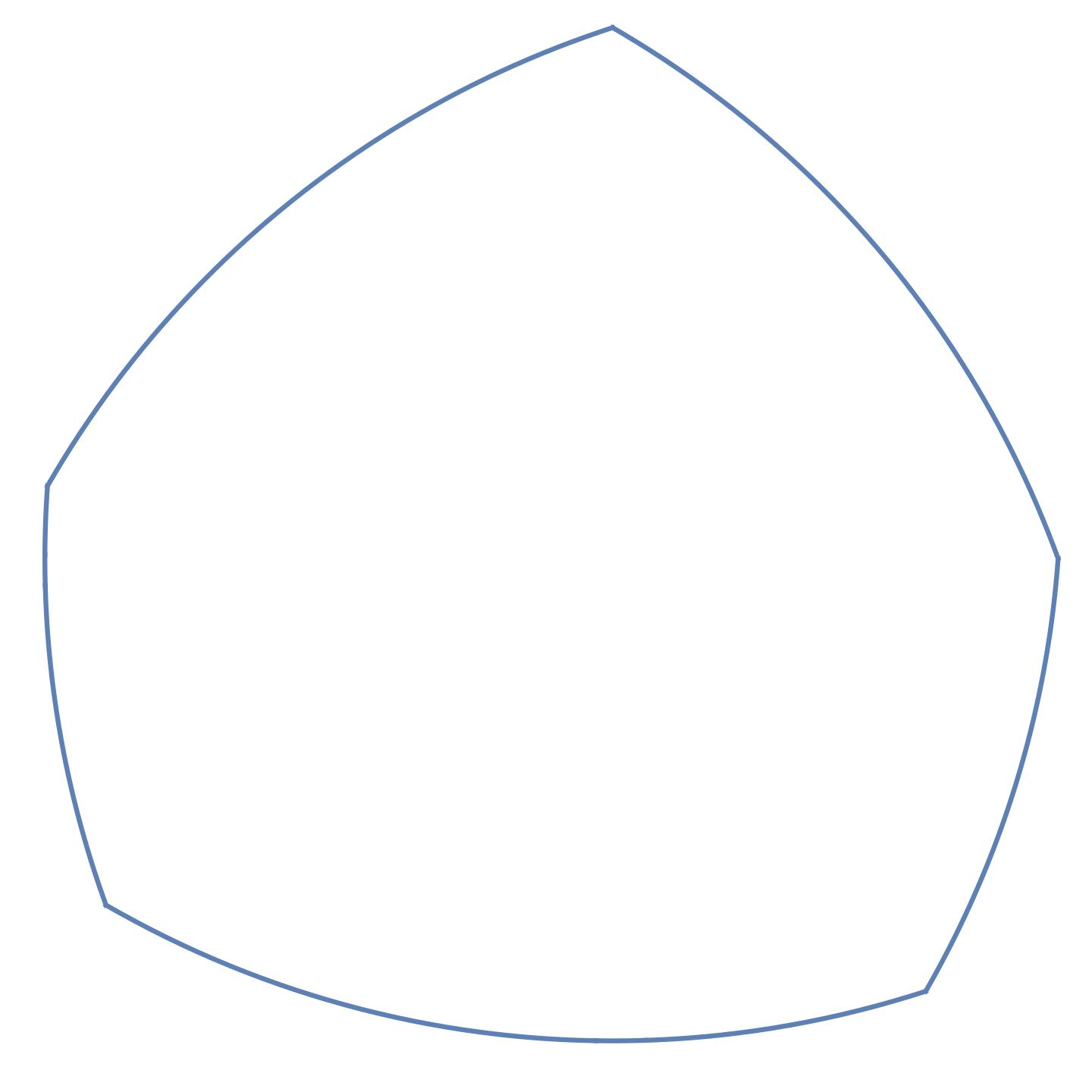}
                   \caption{A Reuleaux pentagon having boundary arcs of differing lengths. Such an example is an irregular Reuleaux pentagon.}
     \end{subfigure}
\end{figure}

\par We realize that many of the considerations in this note can be extended to shapes of constant width in the Euclidean space $\R^3$. Good references for three--dimensional shapes of constant width is the survey by Chakerian and Groemer \cite{MR731106} and the recent monograph by Martini, Montejano, and Oliveros \cite{MR3930585}. However, despite several notable efforts such as \cite{MR1371579,MR2763770,MR2342202}, an analog of the Blaschke--Lebesgue theorem has not been established for constant width shapes in $\R^3$.  As a result, we will focus our attention on planar shapes.

\section{The support function}
In this section, we will study a basic concept used to analyze convex shapes. Suppose $K\subset\R^2$ is compact and convex. For a given $\theta\in \R$, set 
$$
h(\theta)=\max_{x\in K}x\cdot u(\theta),
$$
where $u(\theta)=(\cos(\theta),\sin(\theta)).$ This function $h: \R\rightarrow \R$ is known as the {\it support function} of $K$ since 
the set of $x\in \R^2$ with $x\cdot u(\theta)\le h(\theta)$ is the supporting half-space of $K$ which has outward normal $u(\theta)$. 
\begin{figure}[h]
     \centering
         \includegraphics[width=.6\columnwidth]{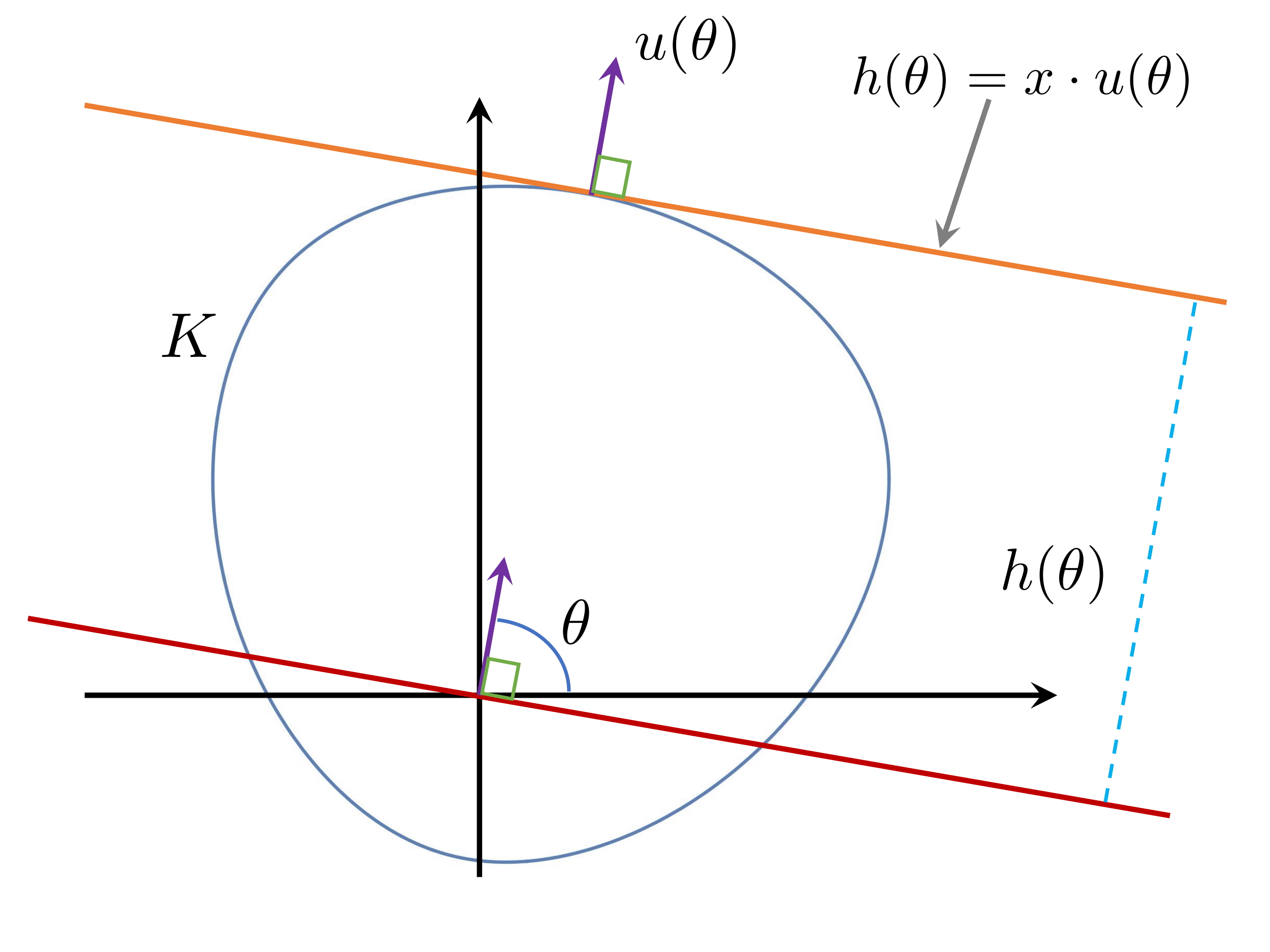}
                   \caption{This figure is the geometric interpretation of the value $h(\theta)$ as the distance from the origin to the line which supports $K$ and has outward normal $u(\theta)$.}
\end{figure}
\par  As $K$ is equal to the intersection of all half-spaces which include $K$, it follows that 
\be\label{KformulaH}
K=\bigcap_{\theta\in \R}\left\{x\in \R^2: x\cdot u(\theta)\le h(\theta)\right\}.
\ee
That is, $x\in K$ if and only if $x\cdot u(\theta)\le h(\theta)$ for all $\theta\in \R$. It particular, a convex shape can be recovered from its support function. It is also not hard to verify that for any $y\in K$, $h(\theta)-y\cdot u(\theta)$ is the distance from $y$ to the supporting line of $K$ with outward normal $u(\theta)$.

\par  Note that $h$ is continuous and $2\pi$--periodic. In the sequel, we'll write $h\in C(\es)$ for any function with these properties.  It will also be useful for us to characterize which functions in $C(\es)$ are support functions.  The inequality in part $(ii)$ of the following proposition was first noted by Kallay \cite{MR350618}. 
\begin{prop}\label{KallayProp}
Suppose $h\in C(\es)$. The following are equivalent. \\
$(i)$ $h$ is the support function of a convex and compact $K\subset \R^2$. \\
$(ii)$ For each $\theta\in \R$ and $\phi\in [-\pi/2,\pi/2]$,
$$
h(\theta+\phi)+h(\theta-\phi)\ge 2h(\theta)\cos(\phi).
$$
$(iii)$ For each smooth $f:\R\rightarrow [0,\infty)$ with compact support, 
\be\label{WeakODE}
\int_{\R}h(\theta)(f''(\theta)+f(\theta))d\theta\ge 0.
\ee
\end{prop}
\begin{proof}
$(i)\Rightarrow (ii)$ Using the angle sum-to-product formulae for sine and cosine, we find 
$$
u(\theta+\phi)+u(\theta-\phi)= 2u(\theta)\cos(\phi). 
$$
Therefore, 
$$
2x\cdot u(\theta)\cos(\phi)=x\cdot u(\theta+\phi)+x\cdot u(\theta-\phi)\le h(\theta+\phi)+h(\theta-\phi)
$$
for $x\in K$.  Since $\cos(\phi)\ge 0$ for $\phi\in [-\pi/2,\pi/2]$, we also have 
$$
2h(\theta)\cos(\phi)=\max_{x\in K}2x\cdot u(\theta)\cos(\phi)\le h(\theta+\phi)+h(\theta-\phi).
$$
$(ii)\Rightarrow (iii)$ Suppose  $f:\R\rightarrow [0,\infty)$ is smooth and has compact support. Then for $\phi\in[-\pi/2,\pi/2]$ with $\phi\neq 0$, 
\begin{align*}
0&\le\frac{1}{\phi^2} \int_{\R}f(\theta)\left( h(\theta+\phi)+h(\theta-\phi)-2h(\theta)\cos(\phi)\right)d\theta\\
&=\frac{1}{\phi^2}\int_{\R}h(\theta)\left( f(\theta+\phi)+f(\theta-\phi)-2f(\theta)\cos(\phi)\right)d\theta\\
&=\int_{\R}h(\theta)\left[\frac{ f(\theta+\phi)+f(\theta-\phi)-2f(\theta)}{\phi^2}\right]+h(\theta)f(\theta)\left[2\frac{(1-\cos(\phi))}{\phi^2}\right]d\theta.
\end{align*}
By our assumptions on $f$, we can send $\phi\rightarrow 0$ in the integral above to get 
$$
0\le \int_{\R}h(\theta)(f''(\theta)+f(\theta))d\theta.
$$
$(iii)\Rightarrow (i)$  For $\epsilon>0$, we consider the mollification of $h$: set
$$
h^\epsilon(\theta)=\int_{\R}\eta^\epsilon(\phi)h(\theta-\phi)d\phi
$$
for $\theta\in \R$. Here $\eta^\epsilon(t)=\epsilon^{-1}\eta(t/\epsilon)$, and $\eta: \R \rightarrow [0,\infty)$ is a smooth symmetric function which is supported in $[-1,1]$ with $\int_{\R}\eta(t)dt=1$.  As $h\in C(\es)$, it is routine to show $h^\epsilon $ converges to $h$ uniformly. Moreover, $h^\epsilon$ is smooth and  
\begin{align*}
(h^\epsilon)''(\theta)+h^\epsilon(\theta)&=\int_{\R}\left[(\eta^\epsilon)''(\phi)+\eta^\epsilon(\phi)\right]h(\theta-\phi)d\phi\\
&=\int_{\R}\left[(\eta^\epsilon)''(\theta-\phi)+\eta^\epsilon(\theta-\phi)\right]h(\phi)d\phi\\
&\ge 0
\end{align*}
for each $\theta\in \R$. 

\par For $u\in \R^2$ with $u\neq 0$, we define 
$$
H^\epsilon(u)=|u|h^\epsilon(\theta),
$$
where $\theta\in\R$ is chosen so that $u/|u|=u(\theta)$; and when $u=0$, we set $H^\epsilon(u)=0$.   Note that $H^\epsilon$ is positively homogeneous and satisfies 
$$
H^\epsilon(u(\theta))=h^\epsilon(\theta)
$$
for each $\theta\in \R$. In particular, $H^\epsilon$ is smooth away from the origin and direct computation yields  $DH^\epsilon(u)\cdot u=H^\epsilon(u)$ and $D^2H^\epsilon(u)u=0$ for $u\neq 0$. It follows that
$$
D^2H^\epsilon(u(\theta))(\alpha u(\theta)+\beta u'(\theta))\cdot (\alpha u(\theta)+\beta u'(\theta))=\beta^2((h^\epsilon)''(\theta)+h^\epsilon(\theta))\ge 0.
$$ 
We conclude $H^\epsilon$ is convex.  Sending $\epsilon\rightarrow 0$, we find that $H^\epsilon$ converges locally uniformly to a positively homogeneous and convex $H: \R^2\rightarrow \R$ which fulfills
$$
h(\theta)=H(u(\theta))
$$
for all $\theta$. 

\par  Define $K$ as in \eqref{KformulaH} with the $h$ we are currently studying, and let $\tilde h$ be the support function of $K$.  Fix $\theta\in \R$. First observe that 
$$
\tilde h(\theta)=\max\{x\cdot u(\theta): x\in K\}\le h(\theta),
$$
as $h(\theta)\ge x\cdot u(\theta)$ for all $x\in K$.  We leave it as an exercise to show that  
$$
H(u(\theta))=p\cdot u(\theta)
$$
for any $p$ belonging to the subdifferential of $H$ at $u(\theta)$. In particular, for any $v\in \R^2$, 
$$
H(v)\ge H(u(\theta))+p\cdot (v-u(\theta))=p\cdot v.
$$
Thus, $h(\theta)=p\cdot u(\theta)$ and $h(\phi)=H(u(\phi))\ge p\cdot u(\phi)$ for all $\phi\in \R$. It follows that $p\in K$ and 
$$
h(\theta)=p\cdot u(\theta)\le \tilde h(\theta).
$$
As a result, $h$ is the support function of $K$. 
\end{proof}
\begin{cor}\label{2ndDeriveEst}
The support function $h$ is twice differentiable for almost every $\theta\in \R$ and 
$$
h''(\theta)+h(\theta)\ge 0
$$
at any such $\theta$. 
\end{cor}
\begin{proof}
By the previous proposition, 
\be\label{2ndDiffQuotient}
\frac{ h(\theta+\phi)-2h(\theta)+h(\theta-\phi)}{\phi^2}\ge 2\left(\frac{\cos(\phi)-1}{\phi^2}\right)h(\theta)
\ee
for all $\theta$ and each $\phi\in [-\pi/2,\pi/2]$ with $\phi\neq 0$. Since $1\ge \cos(\phi)\ge 1-\frac{1}{2}\phi^2$, we find 
$$
\frac{ h(\theta+\phi)-2h(\theta)+h(\theta-\phi)}{\phi^2}\ge -\|h\|_\infty
$$
for all $\theta$ and  $0<|\phi|\le \pi/2$ . 
It follows that there is a constant $b$ for which $h(\theta)+\frac{b}{2}\theta^2$ is convex.  By Alexandrov's theorem, $h$ is twice differentiable for almost every $\theta\in \R$.  If $h$ is twice differentiable at $\theta$, we can send $\phi$ to $0$ in \eqref{2ndDiffQuotient} to get 
$h''(\theta)\ge -h(\theta).$
\end{proof}
\begin{rem}
For a convex and compact $K$ with smooth boundary, $h''(\theta)+h(\theta)$ is the radius of curvature of the boundary point with outward unit normal $u(\theta)$.  
\end{rem}
\subsection{Support function of a constant width curve}
Let us now refine our considerations to constant width curves.  We will use the fact that constant width 
shapes are strictly convex (Theorem 3.1.1 of \cite{MR3930585}). In particular, if $K\subset \R^2$ has constant width, then for each $\theta\in \R$ there is a unique $\gamma(\theta)\in \partial K$ such that 
\be\label{hgammaform}
h(\theta)=\gamma(\theta)\cdot u(\theta).
\ee
It turns out that this implies $h$ is actually continuously differentiable. 
\begin{lem}\label{ConeLemma}
The mapping $\gamma:\R\rightarrow \partial K$ is continuous and 
$$
h'(\theta)=\gamma(\theta)\cdot u'(\theta).
$$
for all $\theta\in \R$. 
\end{lem}
\begin{proof}
Suppose $\theta_k\rightarrow \theta$ as $k\rightarrow\infty$. As $\gamma(\theta_k)\in \partial K$, there is a subsequence $\gamma(\theta_{k_j})$ which converges to some $\xi\in \partial K$.  Moreover, 
$$
h(\theta)=\lim_{j\rightarrow\infty}h(\theta_{k_j})=\lim_{j\rightarrow\infty}\gamma(\theta_{k_j})\cdot u(\theta_{k_j})=\xi\cdot u(\theta).
$$
By uniqueness, $\xi=\gamma(\theta)$. As this limit is independent of the subsequence, $\gamma(\theta_k)\rightarrow \gamma(\theta)$. We conclude that $\gamma$ is continuous.  

\par Fix $\theta\in \R$, and note that as $\gamma(\theta)\in K$, 
$$
h(\theta+\tau)\ge \gamma(\theta)\cdot u(\theta+\tau)
$$
for each $\tau>0$.  Therefore, 
$$
\frac{h(\theta+\tau)-h(\theta)}{\tau}\ge \gamma(\theta)\cdot \frac{u(\theta+\tau)-u(\theta)}{\tau}.
$$
Likewise, we find 
$$
\frac{h(\theta+\tau)-h(\theta)}{\tau}\le \gamma(\theta+\tau)\cdot \frac{u(\theta+\tau)-u(\theta)}{\tau}.
$$
As a result, 
$$
\lim_{\tau\rightarrow0^+}\frac{h(\theta+\tau)-h(\theta)}{\tau}=\gamma(\theta)\cdot u'(\theta).
$$
Virtually the same considerations for $\tau<0$ lead to 
$$
\lim_{\tau\rightarrow0^-}\frac{h(\theta+\tau)-h(\theta)}{\tau}=\gamma(\theta)\cdot u'(\theta).
$$
We conclude that $h'(\theta)$ exists and is equal to $\gamma(\theta)\cdot u'(\theta)$.  
\end{proof}
\par In the proposition below, we will  say $f\in C^{1,1}(\mathbb{S})$ provided $f: \R\rightarrow \R$ is $2\pi$--periodic, $f$ is continuously differentiable, and $f'$ is Lipschitz continuous.   
\begin{prop}\label{ConstantWidthConstraint}
Suppose $K$ has constant width. Then 
$$
h(\theta+\pi)+h(\theta)=1
$$
for all $\theta\in \R$. Moreover, 
$$
0\le h''(\theta)+h(\theta)\le 1
$$
for almost all $\theta\in \R$ and $h\in C^{1,1}(\mathbb{S})$.
\end{prop} 
\begin{proof}
Fix $x\in K$. Recall that the distance from $x$ to the supporting line with outward normal $u(\theta)$ is $h(\theta)-x\cdot u(\theta)$. Likewise, the distance from $x$ to the supporting line with outward normal $u(\theta+\pi)$ is $h(\theta+\pi)-x\cdot u(\theta+\pi)$. Since 
$
u(\theta+\pi)=-u(\theta),
$
the distance between these supporting lines is equal to 
$$
h(\theta+\pi)-x\cdot u(\theta+\pi)+h(\theta)-x\cdot u(\theta)=h(\theta+\pi)+h(\theta).
$$
As $K$ has constant width, it must be that $h(\theta+\pi)+h(\theta)=1$ for all $\theta$.  

\par Select a $\theta$ for which $h''(\theta)$ exists. As $h(\phi+\pi)=1-h(\phi)$ for all $\phi$, $h$ is twice differentiable at $\theta+\pi$, as well.   Moreover, 
$$
h''(\theta+\pi)=-h''(\theta)
$$
and 
$$
h''(\theta)+h(\theta)=1-[h''(\theta+\pi)+h(\theta+\pi)]\le 1. 
$$
Moreover, since
$$
-h(\theta)\le h''(\theta)\le 1-h(\theta)
$$
for almost every $\theta\in \R$, $h''$ is an essentially bounded function. As a result, $h'$ is Lipschitz continuous. 
\end{proof}
\begin{rem}
The inequality $h''(\theta)+h(\theta)\le 1$ implies that the curvature of a smooth constant width curve is always greater than or equal to one.  
\end{rem}

\par The subsequent assertion is a converse to some of the facts derived above. It is also a useful tool in generating shapes of constant width. 
\begin{prop}\label{KallayProp} 
Suppose $h: \mathbb{S}\rightarrow \R$ satisfies
$$
\begin{cases}
h\in C^{1,1}\left(\mathbb{S}\right)\\
h(\theta+\pi)+h(\theta)=1 \textup{ for all $\theta\in \R$}\\
h''(\theta)+h(\theta)\ge0 \textup{ for almost every $\theta\in \R$}.
\end{cases}
$$
Then $h$ is the support function of the constant width shape $K$ defined in \eqref{KformulaH}.  
\end{prop}
\begin{proof}
Since $h\in C^{1,1}\left(\mathbb{S}\right)$, $h'$ is Lipschitz continuous.  Thus, for smooth $f:\R\rightarrow [0,\infty)$ with compact support we may integrate $h(\theta)f''(\theta)$ by parts twice to find
$$
\int_{\R}h(\theta)(f''(\theta)+f(\theta))d\theta =\int_{\R}(h''(\theta)+h(\theta))f(\theta)d\theta\ge 0.
$$
As explained in our proof of Proposition \ref{KallayProp}, $h$ is the support function the convex and compact $K$ defined in \eqref{KformulaH}.  Since $h(\theta+\pi)+h(\theta)=1 $ for all $\theta$, $K$ has constant width. 
\end{proof}

\par Let us study a few examples. 
\begin{ex}[Circles]
The support function of the circle of radius $1/2$ centered at $a\in \R^2$ is given by 
$$
h(\theta)=\max_{|x-a|\le 1/2}x\cdot u(\theta)=\max_{|x-a|\le 1/2}(x-a)\cdot u(\theta)+a\cdot u(\theta)=\frac{1}{2}+a\cdot u(\theta).
$$
\end{ex}
\begin{ex}[Reuleaux triangle]
Suppose $K$ is a Reuleaux triangle with vertices 
$$
\displaystyle \left(\frac{1}{2},\frac{1}{2\sqrt{3}}\right),\quad
\displaystyle \left(-\frac{1}{2},\frac{1}{2\sqrt{3}}\right),\quad\text{and}\quad
\displaystyle \left(0,-\frac{1}{\sqrt{3}}\right).
$$
If $u(\theta)$ is an outward normal to $\partial K$ at a vertex $a$, then $h(\theta)=u(\theta)\cdot a$. Furthermore, if 
$u(\theta)$ is an outward normal to $\partial K$ at a circular arc centered at $b$, then $h(\theta)=1+u(\theta)\cdot b$. Combining these observations with some elementary 
case analysis leads to the following expression for the support function of $K$. For $\theta\in [0,2\pi]$,
$$
h(\theta)=
\begin{cases}
\displaystyle u(\theta)\cdot \left(\frac{1}{2},\frac{1}{2\sqrt{3}}\right),\quad &\displaystyle 0\le \theta\le \frac{\pi}{3}\\\\
\displaystyle 1+u(\theta)\cdot\left(0,-\frac{1}{\sqrt{3}}\right) ,\quad & \displaystyle\frac{\pi}{3}\le \theta\le\frac{2\pi}{3}\\\\
\displaystyle u(\theta)\cdot \left(-\frac{1}{2},\frac{1}{2\sqrt{3}}\right),\quad &\displaystyle \frac{2\pi}{3}\le\theta\le \pi\\\\
\displaystyle 1+u(\theta)\cdot \left(\frac{1}{2},\frac{1}{2\sqrt{3}}\right),\quad & \displaystyle\pi \le \theta\le\frac{4\pi}{3}\\\\
\displaystyle u(\theta)\cdot\left(0,-\frac{1}{\sqrt{3}}\right),\quad &\displaystyle \frac{4\pi}{3}\le\theta\le \frac{5\pi}{3}\\\\
\displaystyle 1+u(\theta)\cdot \left(-\frac{1}{2},\frac{1}{2\sqrt{3}}\right),\quad & \displaystyle\frac{5\pi}{3}\le\theta\le 2\pi.
\end{cases}
$$
\end{ex}
\begin{ex}[regular Reuleaux polygons]\label{RegReulEx}
We can build on our example above to express the support function of a {\it regular} Reuleaux polygon. This is a Reuleaux polygon in which the lengths of the circular arcs forming the boundary are all equal.  Suppose $N\ge 3$ with $N$ odd, and set  
$$
x_k=\left(\frac{\sin\left(\frac{k\pi}{N}\right)-\sin\left(\frac{(k-1)\pi}{N}\right)}{2\sin\left(\frac{\pi}{N}\right)},-\frac{\cos\left(\frac{k\pi}{N}\right)-\cos\left(\frac{(k-1)\pi}{N}\right)}{2\sin\left(\frac{\pi}{N}\right)}\right)
$$
for $k=1,2,\dots, 2N$.  Next define
$$
h(\theta)=
\begin{cases}
 x_k\cdot u(\theta), \; &\text{ for $\theta\in\left[\frac{(k-1)\pi}{N},\frac{k\pi}{N}\right]$ and $k=1,3,\dots, 2N-1$}\\\\
1-x_k\cdot u(\theta), \; & \text{ for $\theta\in\left[\frac{(k-1)\pi}{N},\frac{k\pi}{N}\right]$ and $k=2,4,\dots, 2N$}
\end{cases}
$$
for $\theta\in [0,2\pi]$ and extend this function $2\pi$-periodically to all of $\R$.  We leave it as an exercise to verify that $h$ is a support function of the $N$-sided regular Reuleaux polygon with vertices $x_1,x_3,\dots, x_{2N-1}$. 
The Reuleaux triangle described in the example above is the case $N=3$. We've also displayed the case $N=5$ in Figure \ref{PentFig}.
\end{ex}
\begin{ex}[Perturbation of a circle]
Using Proposition \ref{KallayProp}, we can design a curve of constant width starting with any $g\in C^{1,1}(\es)$ which satisfies 
$$
g(\theta+\pi)=-g(\theta)
$$
for all $\theta\in \R$.  For $\delta>0$, set 
$$
h(\theta)=\frac{1}{2}+\delta g(\theta)
$$
for $\theta\in \R$. Note that  $h\in C^{1,1}(\es)$ and 
$$
h(\theta)+h(\theta+\pi)=1
$$
for all $\theta\in \R$.  Also note for almost every $\theta\in \R$,  
$$
h''(\theta)+h(\theta)=\frac{1}{2}+\delta (g''(\theta)+g(\theta))\ge 0
$$
provided $\delta$ is chosen sufficiently small. By Proposition \ref{KallayProp}, $h$ is the support function of a constant width curve.  
\begin{rem} We used this method to create the curve in Figure \ref{FirstFig}. For that specific example, we chose $\delta=1/160$ and $g(\theta)=\cos(3\theta)+\sin(7\theta)$. 
\end{rem}
\end{ex}
\begin{figure}[h]
     \centering
     \begin{subfigure}[b]{0.32\textwidth}
         \centering
         \includegraphics[width=\textwidth]{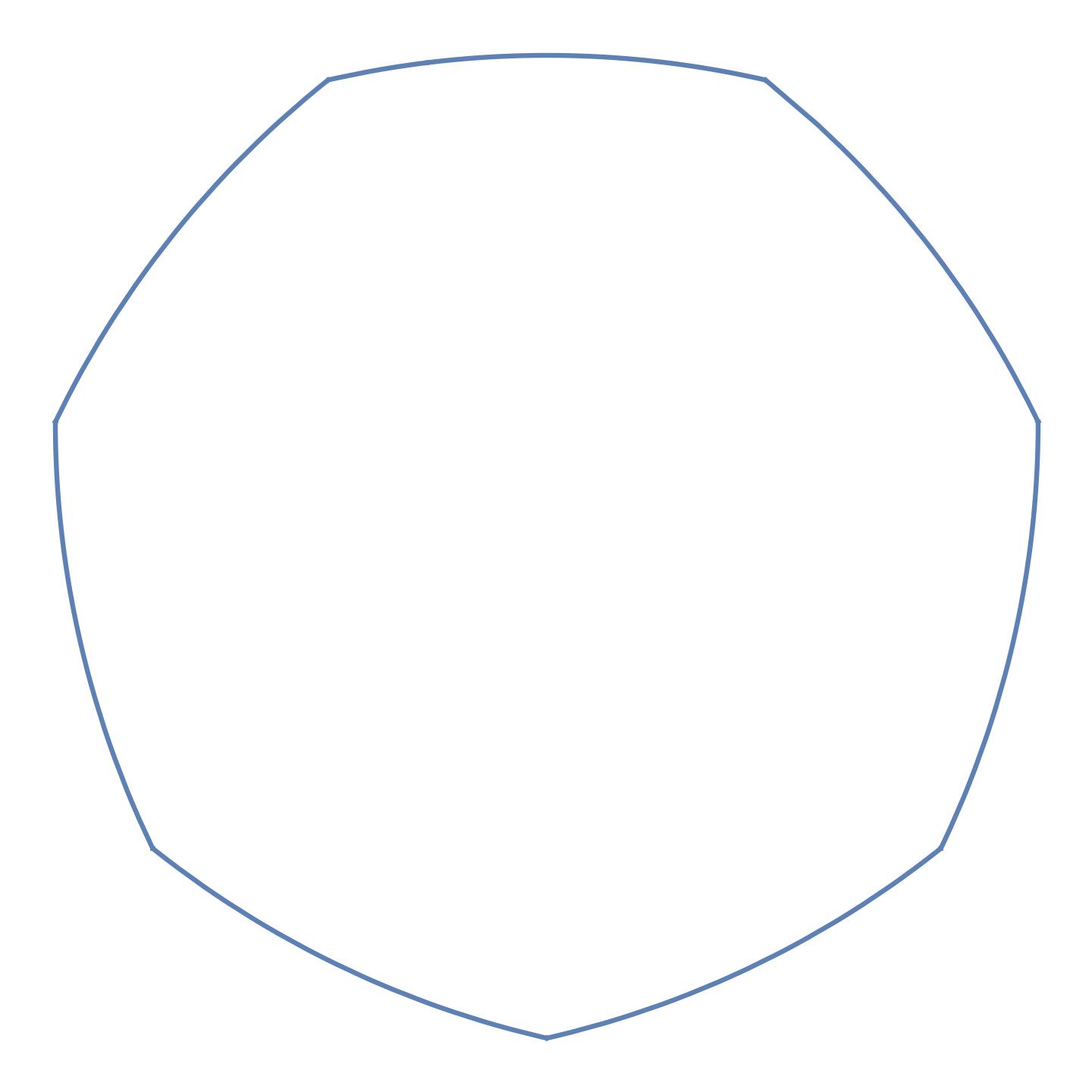}
     \end{subfigure}
     \begin{subfigure}[b]{0.32\textwidth}
         \centering
         \includegraphics[width=\textwidth]{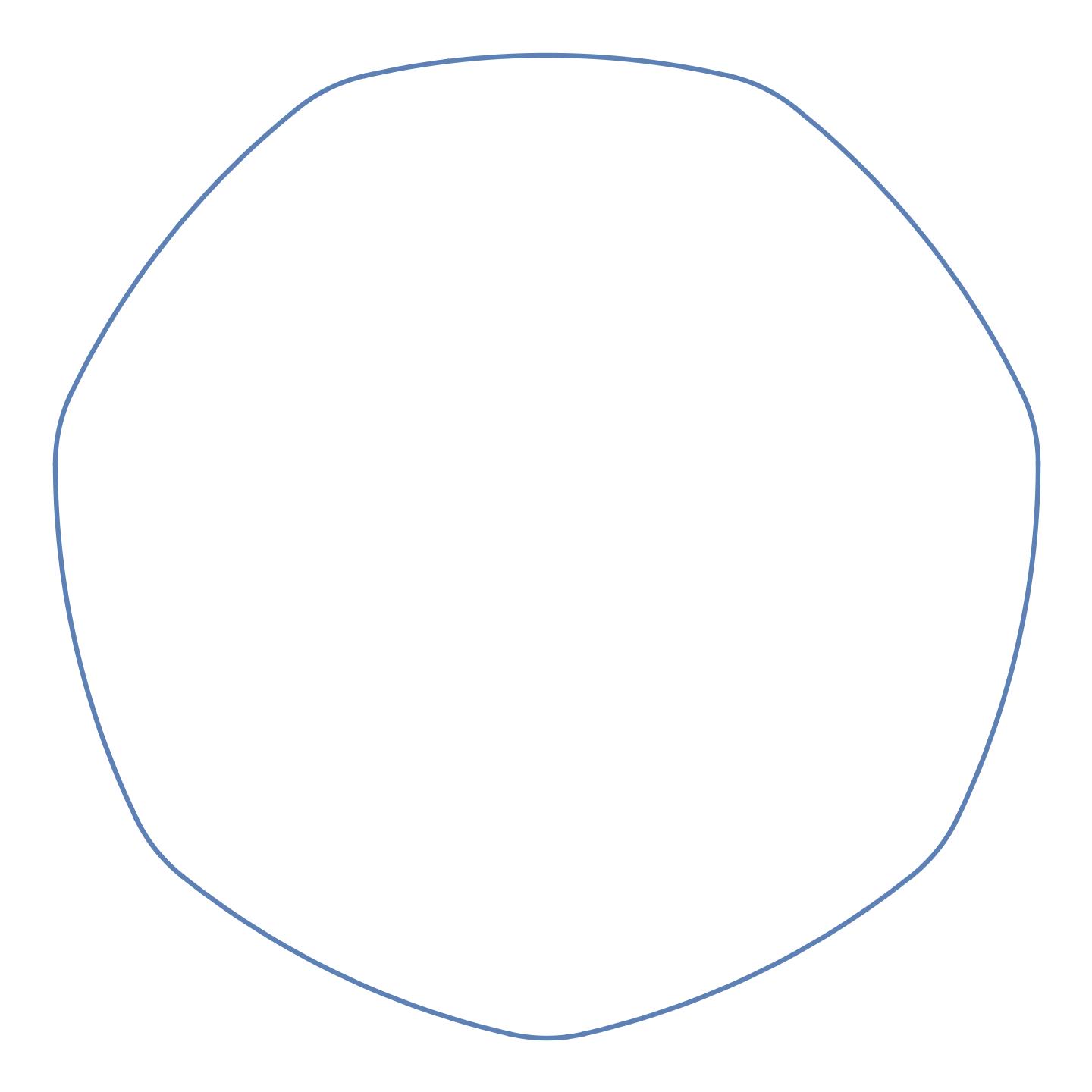}
     \end{subfigure}
      \begin{subfigure}[b]{0.32\textwidth}
         \centering
         \includegraphics[width=\textwidth]{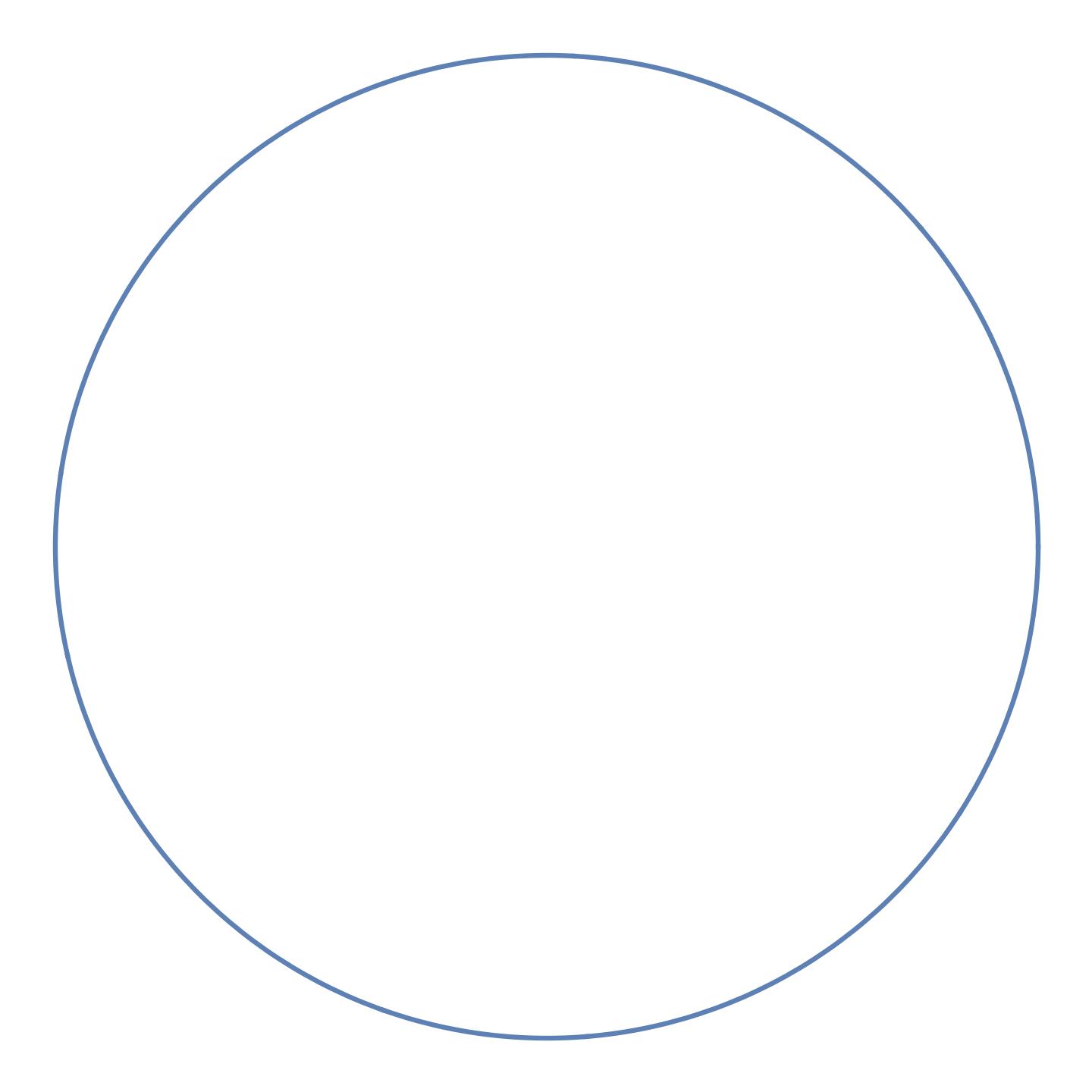}
     \end{subfigure}
      \caption{On the left is a regular Reuleaux heptagon $K_1$, and on the right is a circle $K_2$ of radius 1/2; the middle curve is the convex combination $(4/5) K_1+(1/5)K_2$.}\label{ConvCombFig}
\end{figure}
\begin{ex}[convex combinations]
Examples of constant width shapes can also be designed by forming convex combination of constant width shapes. In particular, it is routine to verify 
that if $K_1,K_2\subset \R^2$ are constant width shapes and $\lambda\in [0,1]$, the support function of 
$$
(1-\lambda)K_1+ \lambda K_2:=\{(1-\lambda)x_1+ \lambda x_2\in \R^2: x_1\in K_1, x_2\in K_2\}
$$
is $(1-\lambda)h_1+\lambda h_2.$ As $(1-\lambda)h_1+\lambda h_2$ satisfies the hypotheses 
of Proposition \ref{KallayProp},   $(1-\lambda)K_1+\lambda K_2$ is also a constant width shape.  See Figure \ref{ConvCombFig} for an example.
\end{ex}

\subsection{Parametrization of a constant width curve}
It is possible to parametrize the boundary curve of a constant width $K$ using the support function of $K$.  To do so, we will write an explicit formula for $\gamma(\theta)$ discussed in \eqref{hgammaform} and identify a few properties of this path. 
\begin{prop}\label{gammaMultiPart}
Suppose $K\subset \R^2$ has constant width, $h$ is the support function of $K$, and define $\gamma: \mathbb{\R}\rightarrow \partial K$ via \eqref{hgammaform}.  
(i) For $\theta\in \R$,  
$$
\gamma(\theta)=h(\theta)u(\theta)+h'(\theta)u'(\theta).
$$
(ii) For $\theta\in \R$, 
$$
\gamma(\theta+\pi)=\gamma(\theta)-u(\theta)
$$
(iii) $\gamma$  is surjective.\\
(iv) $\gamma$ is injective on any interval $(\theta_0,\theta_1)\subset [0,2\pi]$ for which 
\be\label{injectivityCond}
h''(\theta)+h(\theta)>0\text{ for almost every $\theta\in (\theta_0,\theta_1)$}. 
\ee
(v) For $\theta_1,\theta_2\in \R$,
\be\label{gammaLipEst}
|\gamma(\theta_1)-\gamma(\theta_2)|\le |\theta_1-\theta_2|.
\ee
\end{prop}
\begin{proof}
$(i)$ As $\{u(\theta), u'(\theta)\}$ is an orthonormal basis of $\R^2$, 
$$
\gamma(\theta)=[\gamma(\theta)\cdot u(\theta)]u(\theta)+[\gamma(\theta)\cdot u'(\theta)]u'(\theta)=h(\theta)u(\theta)+h'(\theta)u'(\theta).
$$
$(ii)$ This follows directly from the constant width condition $h(\theta+\pi)+h(\theta)=1$. 
$(iii)$ For each $x\in \partial K$, there is at least one supporting plane for $K$ which includes $x$. It follows that there is $\theta$ such that $h(\theta)=x\cdot u(\theta)$. This in turn implies $\gamma(\theta)=x$.  
  \par $(iv)$ Suppose that $\gamma(\phi_0)=\gamma(\phi_1)=:x\in \partial K$ for $\phi_0,\phi_1\in (\theta_0,\theta_1)$ with $\phi_0<\phi_1$. Then $x\in \partial K$ has two distinct supporting lines $h(\phi_0)=y\cdot u(\phi_0)$ and 
$h(\phi_1)=y\cdot u(\phi_1)$. It follows that $h(\theta)=x\cdot u(\theta)$ for $\theta\in (\phi_1,\phi_1)$, which contradicts \eqref{injectivityCond}.
As a result, $\gamma$ is injective on $(\theta_0,\theta_1)$. 
\par $(v)$ Direct computation gives  
\be\label{gammaDerivative}
\gamma'(\theta)=(h''(\theta)+h(\theta))u'(\theta)
\ee
for almost every $\theta\in \R$. Since $0\le h''(\theta)+h(\theta)\le 1$, $|\gamma'(\theta)|\le 1$ for almost every $\theta$. Therefore, 
$|\gamma(\theta_1)-\gamma(\theta_2)|\le |\theta_1-\theta_2|$ for all $\theta_1,\theta_2\in \R$. 
\end{proof}

A nice consequence of the above proposition is following. 
\begin{thm}[Barbier's theorem]\label{Barbier}
The perimeter of a constant width curve is equal to $\pi$. 
\end{thm}
\begin{proof}
Let $h$ be a support function of a constant width curve and $\gamma$ the corresponding parametrization discussed in the previous proposition. In view of \eqref{gammaDerivative}, the perimeter of the curve is
$$
\int^{2\pi}_0| \gamma'(\theta)|d\theta=\int^{2\pi}_0(h''(\theta)+h(\theta))d\theta=\int^{2\pi}_0h(\theta)d\theta=\frac{1}{2}\int^{2\pi}_0(h(\theta+\pi)+h(\theta))d\theta=\pi.
$$
\end{proof}
\begin{cor}
Among all curves of constant width, circles of radius 1/2 enclose the most area.  Moreover, circles are the only curves of constant width attaining the maximum possible area. 
\end{cor}
\begin{proof}
Suppose $K\subset \R^2$ is constant width shape and that $K$ encloses area $A$. Barbier's theorem implies that the perimeter of $K$ is equal to $\pi$. According to the isoperimetric inequality, $ A\le \pi/4$ and equality holds if and only if $K$ is a circle. We conclude as any circle of radius $1/2$ has constant width and area $A=\pi/4$.
\end{proof}
It will also be useful to express the area of a constant width shape in terms of the support function.  We will use $A(K)$ to denote the area of a convex and compact $K\subset \R^2$. 
\begin{prop}
Suppose $K\subset \R^2$ has constant width and $h$ is the support function of $K$. Then 
$$
A(K)=\frac{1}{2}\int^{2\pi}_0h\left(h''+h\right)d\theta.
$$ 
\end{prop}
\begin{proof}
We will employ the parametrization $\gamma=(\gamma^1,\gamma^2)$ discussed above. As $\gamma$ is Lipschitz continuous and parametrizes $\partial K$ counterclockwise, Green's theorem gives 
$$
A(K)=\frac{1}{2}\int^{2\pi}_0\left(\gamma^1(\gamma^2)'-\gamma^2(\gamma^1)'\right)d\theta=\frac{1}{2}\int^{2\pi}_0h\left(h''+h\right)d\theta.
$$
\end{proof}
\begin{rem}
It is sometimes useful to integrate by parts and express the area of $K$ as 
\be\label{IntByPartsArea}
A(K)=\frac{1}{2}\int^{2\pi}_0(h^2-h'^2)d\theta.
\ee
\end{rem}
\begin{ex}\label{AreaRegReul}
Suppose $N\ge 3$ is odd and $K$ is the $N$--sided regular Reuleaux triangle with support function $h$ given in example \ref{RegReulEx}. The area of $K$ is 
\begin{align*}
\frac{1}{2}\int^{2\pi}_0h(h''+h)d\theta&=\frac{1}{2}\sum^{2N}_{k=1}\int^{k\pi/N}_{(k-1)\pi/N}h(h''+h)d\theta\\
&=\frac{1}{2}\sum_{\text{$k$ even}}\int^{k\pi/N}_{(k-1)\pi/N}(1-x_k\cdot u(\theta)) d\theta\\
&=\frac{1}{2}\sum_{\text{$k$ even}}\left\{\frac{\pi}{N}-\frac{\left(\sin\left(\frac{k\pi}{N}\right)-\sin\left(\frac{(k-1)\pi}{N}\right)\right)^2+\left(\cos\left(\frac{k\pi}{N}\right)-\cos\left(\frac{(k-1)\pi}{N}\right)\right)^2}{2\sin\left(\frac{\pi}{N}\right)}\right\}\\
&=\frac{1}{2}\sum_{\text{$k$ even}}\left\{\frac{\pi}{N}-\frac{1-\cos\left(\frac{\pi}{N}\right)}{\sin\left(\frac{\pi}{N}\right)} \right\}\\
&=\frac{\pi}{2}\left(1-\frac{1-\cos\left(\frac{\pi}{N}\right)}{\frac{\pi}{N}\sin\left(\frac{\pi}{N}\right)}\right).
\end{align*}  
It is routine to check that this expression increases in $N$. Therefore, the Reuleaux triangle has the least area among all regular Reuleaux polygons. 
\end{ex}

\section{Approximation by Reuleaux polygons} 
A crucial step in our proof of the Blaschke--Lebesgue theorem is that each shape of constant width can be closely approximated by a Reuleaux polygon.   The following assertion is originally due to Blaschke \cite{MR1511839}.  See also Theorem 6 of Kallay's paper \cite{MR350618} for a related assertion.  
\begin{prop}\label{ApproxProp}
Suppose $K$ is a constant width shape with support function $h$ and $\epsilon>0$. There is a Reuleaux polygon $K_\epsilon$ with support function $h_\epsilon$ such that 
$$
|h(\theta)-h_\epsilon(\theta)|\le \epsilon\quad\text{and}\quad|h'(\theta)-h'_\epsilon(\theta)|\le\epsilon
$$
for each $\theta\in \R$. 
\end{prop}
\begin{proof}
1. By replacing $h$ with 
$$
h_\delta=\frac{h+\delta}{1+2\delta}
$$
for  $\delta>0$ and small, we may suppose that the corresponding parametrization $\gamma$ is injective.  Indeed, it is routine to check that $h_\delta$ satisfies the hypotheses of Proposition \ref{KallayProp} with 
$$
h''_\delta(\theta)+h_\delta(\theta)=\frac{h''(\theta)+h(\theta)+2\delta}{1+2\delta}\ge \frac{2\delta}{1+2\delta}>0
$$
for almost every $\theta$. Part $(iv)$ of Proposition \ref{gammaMultiPart} gives that the parametrization associated with $h_\delta$ is injective.  Moreover, 
$$
|h(\theta)-h_\delta(\theta)|\le \frac{\delta}{1+2\delta}|2h(\theta)-1|
$$ 
and
$$
|h'(\theta)-h'_\delta(\theta)|\le \frac{2\delta}{1+2\delta}|h'(\theta)|
$$ 
for all $\theta$. Since $h$ and $h'$ are bounded functions, $h_\delta$ is a $C^1$ approximation of $h$ with the desired properties mentioned above. Consequently, we will suppose that $\gamma$ is injective. 
\par 2. Suppose $n\in \N$ with 
$$
\frac{\pi}{n}\le \epsilon 
$$
and set 
$$
\theta_i=\frac{i \pi}{n}
$$
for $i=0,\dots, n$.  For each $i=1,\dots, n$, we consider the 4--tuple of points 
$$
\left\{\gamma(\theta_i), \gamma(\theta_{i-1}), \gamma(\theta_i+\pi), \gamma(\theta_{i-1}+\pi)\right\}\subset \partial K.
$$
By our assumption that $\gamma$ is injective, these are four distinct points.  
\begin{figure}[h]
         \centering
         \includegraphics[width=.7\textwidth]{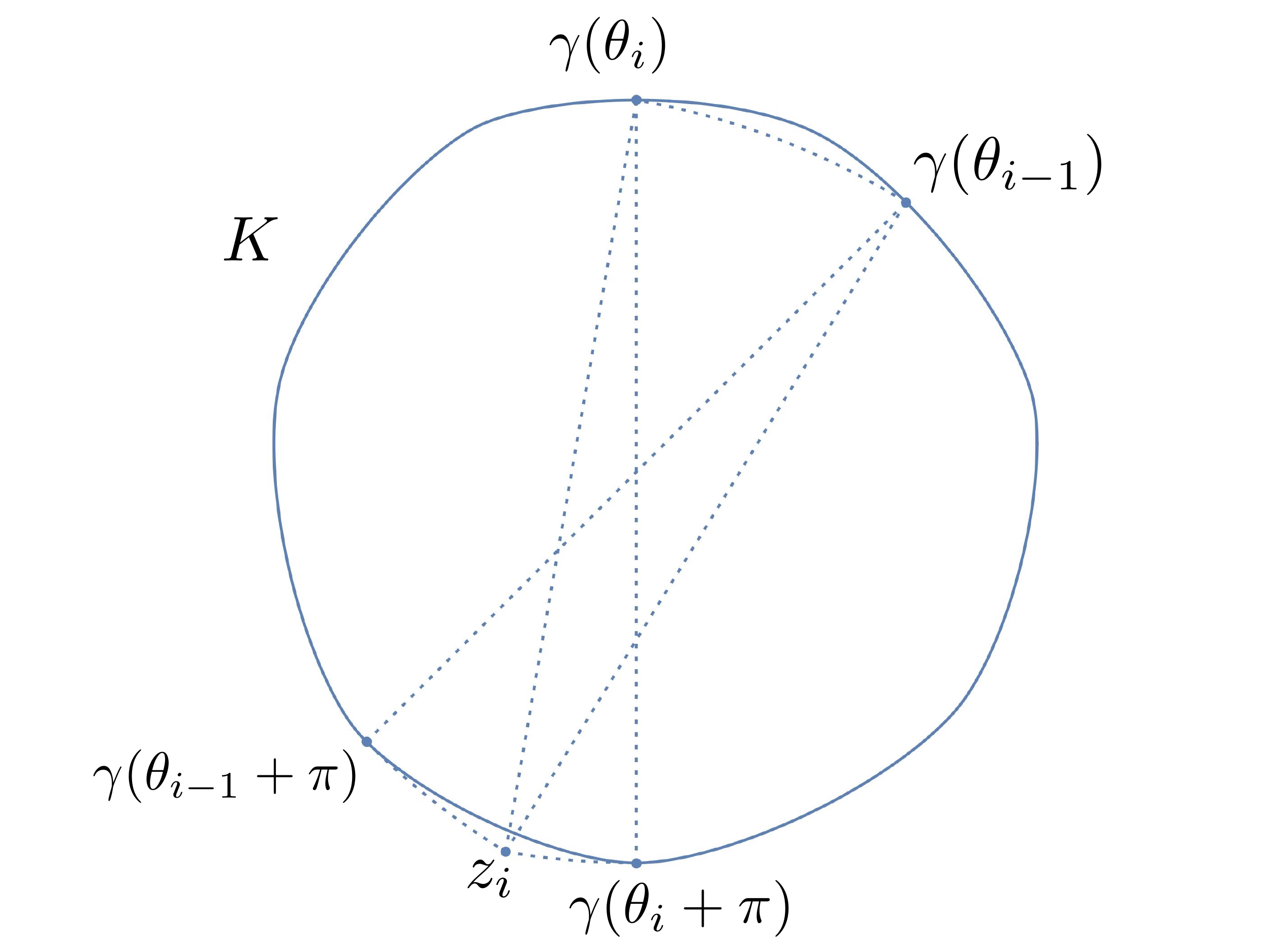}
                   \caption{This diagram illustrates how we can construct a Reuleaux polygon which approximates $K$. We simply partition the interval $0=\theta_0<\dots<\theta_n=\pi$ and replace the part of $\partial K $ between $\gamma(\theta_{i-1})$ and $\gamma(\theta_i)$ with an arc of a circle as shown in this diagram. Then we choose vertices on the other side of $K$ to ensure the resulting curve has constant width.}
                   \label{DensityDiag}
\end{figure}

\par There are two solutions $z\in \R^2$ for which 
$$
|z-\gamma(\theta_i)|=|z-\gamma(\theta_{i-1})|=1.
$$
Let $z_i$ be the solution $z$ which additionally satisfies $|z-\gamma(\theta_i+\pi)|\le 1$ and $|z- \gamma(\theta_{i-1}+\pi)|\le 1$.
See Figure \ref{DensityDiag}. Also observe that 
\be\label{zeyeConditions}
\gamma(\theta_i)-z_i=u(\psi_i)\quad\text{and}\quad \gamma(\theta_{i-1})-z_i=u(\phi_{i})
 \ee
for angles $\phi_i$ and $\psi_i$ with
 $$
 \theta_{i-1}\le \phi_{i}\le \psi_i \le \theta_{i}.
 $$
 
 \par 3. Define
 $$
 h_\epsilon(\theta)
 =\begin{cases}
 \gamma(\theta_{i-1})\cdot u(\theta),  &\theta\in [\theta_{i-1}, \phi_{i}]\\
 1+z_i\cdot u(\theta),  &\theta\in [\phi_{i}, \psi_{i}]\\
\gamma(\theta_i)\cdot u(\theta),  & \theta\in [\psi_{i}, \theta_{i}]
 \end{cases}
 $$
 for $\theta\in [0,\pi]$ and extend $h_\epsilon$ to $[\pi,2\pi]$ by setting 
 $$
 h_\epsilon(\theta+\pi)=1-h_\epsilon(\theta)\quad \theta\in [0,\pi].
 $$
It is straightforward to employ \eqref{zeyeConditions} and show $h_\epsilon$ extends to a $2\pi-$periodic function which is continuously differentiable on $\R$.  As $h''_\epsilon+h_\epsilon$ alternatives between $0$ and $1$ on successive intervals, $h_\epsilon$ is the support function of a Reuleaux polygon. 

\par Suppose $\theta\in [0,\pi]$ and choose $i=1,\dots, n$ such that $\theta\in [\theta_{i-1},\theta_i]$. If $\theta\in [\theta_{i-1},\phi_i]$, then 
\begin{align*}
|h(\theta)-h_\epsilon(\theta)|&=|\gamma(\theta)\cdot u(\theta)-\gamma(\theta_{i-1})\cdot u(\theta)|\\
&=|(\gamma(\theta)-\gamma(\theta_{i-1}))\cdot u(\theta)|\\
&\le |\gamma(\theta)-\gamma(\theta_{i-1})|\\
&\le \theta-\theta_{i-1}\\
&\le \frac{\pi}{n}\\
&\le \epsilon.
\end{align*}
Here we used the Lipschitz estimate \eqref{gammaLipEst}.  By virtually the same argument, $|h(\theta)-h_\epsilon(\theta)|\le \epsilon$ when $\theta\in [\psi_{i},\theta_i]$.  Moreover, if $\theta\in [\phi_{i},\psi_i]$,
\begin{align*}
|h(\theta)-h_\epsilon(\theta)|&=|\gamma(\theta)\cdot u(\theta)-(1+z_i\cdot u(\theta))|\\
&=|\gamma(\theta)\cdot u(\theta)-(u(\theta)+z_i)\cdot u(\theta))|\\
&\le |\gamma(\theta)-(u(\theta)+z_i)|\\
&\le|\gamma(\theta)-\gamma(\theta_i)|+ |\gamma(\theta_i)-(u(\theta)+z_i)|\\
&=|\gamma(\theta)-\gamma(\theta_i)|+ |u(\psi_i)-u(\theta)|\\
&\le \theta_i-\theta+\theta-\psi_i\\
&\le \frac{\pi}{n}\\
&\le \epsilon.
\end{align*}
We conclude 
$$
|h(\theta)-h_\epsilon(\theta)|\le\epsilon
$$
for $\theta\in [0,\pi]$. Since $h(\theta+\pi)=1-h(\theta)$ and $h_\epsilon(\theta+\pi)=1-h_\epsilon(\theta)$ for $\theta\in [0,\pi]$, the estimate above also holds for all $\theta\in [0,2\pi]$.  Finally, we note that the bound 
$$
|h'(\theta)-h'_\epsilon(\theta)|\le\epsilon
$$
for all $\theta\in [0,2\pi]$ follows very similarly. We leave the details to the reader. 
\end{proof}

\begin{figure}[h]
         \centering
         \includegraphics[width=.7\textwidth]{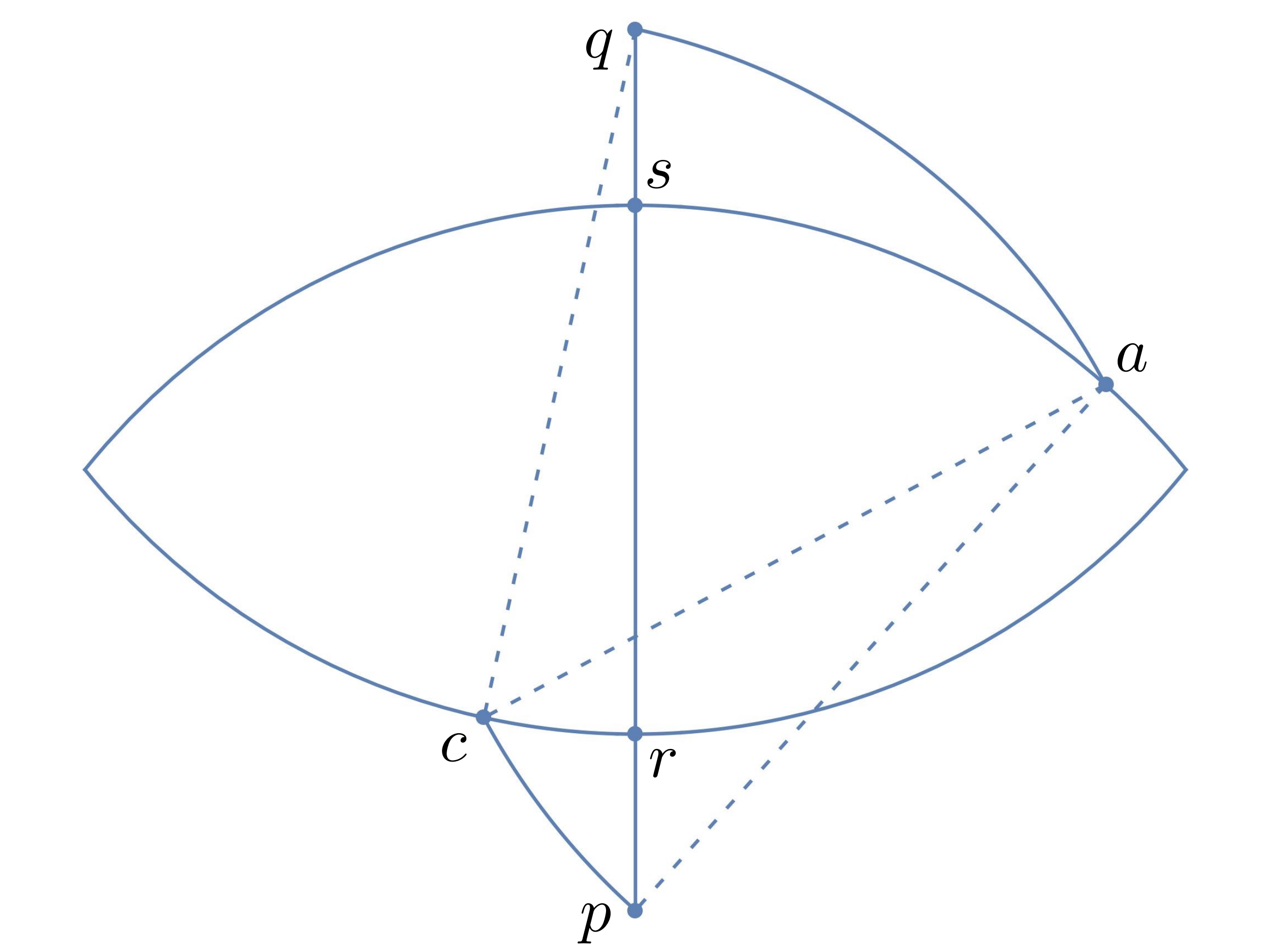}
                   \caption{This is the reference which accompanies Lemma \ref{GeometricLemma}. Lemma \ref{GeometricLemma} asserts that the area of the curvilinear triangle  $\Delta(asq)$ with vertices $a,s,$ and $q$ is at least as much as the curvilinear triangle $\Delta(crp)$ with vertices $c,r,$ and $p$ provided that  the length of arc joining $a$ and $s$ is no less than the length of the arc joining $c$ and $r$.}\label{Lemma42Fig}
\end{figure}
\par In addition, we will argue that for any (possibly irregular) Reuleaux polygon the Reuleaux triangle has least area. This assertion was verified in the solution to problem 7.20 in \cite{MR0123962} and we shall follow this solution closely below.  To this end, we will first need to establish a few technical lemmas. Let us denote $C(x)$ and $D(x)$ for the circle and open disk of radius one centered at $x$, respectively. If $y,z\in C(x)$ satisfy $|y-z|<2$, we will write  $\overset{\frown}{yz}$ for the shorter segment within $C(x)$ which joins $y$ and $z$; by abuse of notation, we will also write  $\overset{\frown}{yz}$  for the length of this arc. In addition, $\Delta(abc)$ will denote a curvilinear triangle bounded by line segments or arcs of circles of radius one with vertices given by $a,b$ and $c$. 
\begin{lem}\label{GeometricLemma}
Assume $p,q,r,s\in \R^2$ with $|p-q|<2,$
and that $s$ and $r$ are on the line segment between $p$ and $q$ with 
$$
|s-p|=|r-q|=1.
$$ 
Suppose $a\in C(p)\cap D(q)$ and $c\in C(q)\cap D(p)$ with
$$
|a-c|=1.
$$
(i) If $\overset{\frown}{as}\ge \overset{\frown}{cr}$, then $A(\Delta(asq))\ge A(\Delta(crp))$. Refer to Figure \ref{Lemma42Fig}.  \\
(ii) The area difference $A(\Delta(asq))- A(\Delta(crp))$ is nondecreasing in the length difference $\overset{\frown}{as}-\overset{\frown}{cr}$.
\end{lem}
\begin{proof}
$(i)$ Since $|s-p|=|r-q|=1$ and $q,s,r,p$ are collinear, $|p-q|=|q-s|+1=1+|r-p|$. Thus, $|q-s|=|r-p|$. As $\overset{\frown}{as}\ge \overset{\frown}{cr}$, we can place a curvilinear triangle which is congruent to $\Delta(crp)$ within $\Delta(asq)$. Consequently, $A(\Delta(asq))\ge A(\Delta(crp))$.  See Figure \ref{SecondYB}.  
\begin{figure}[h]
\centering
         \includegraphics[width=.7\textwidth]{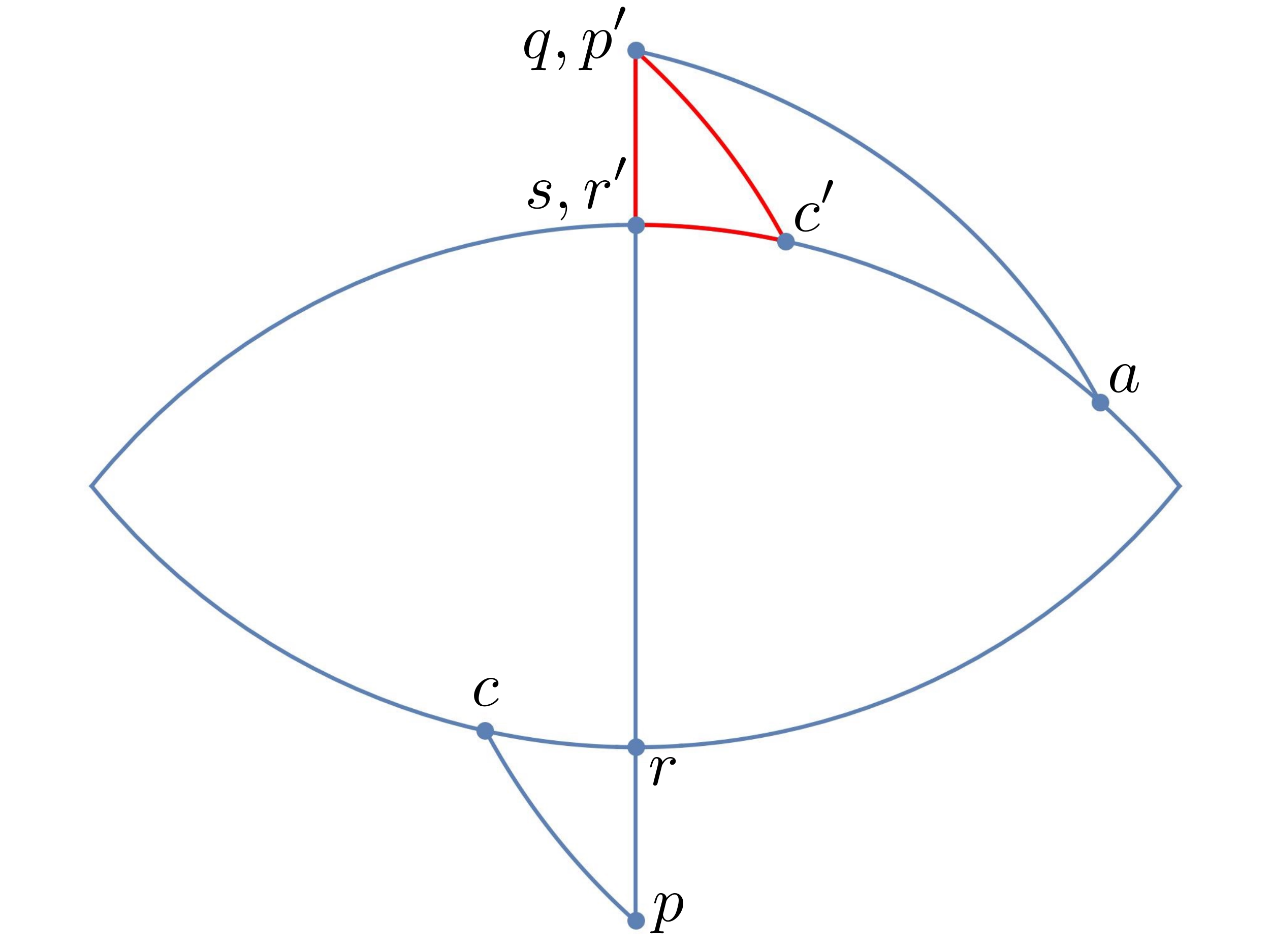}
                   \caption{Here we take a curvilinear triangle $\Delta(c'r'p')$ which is congruent to $\Delta(crp)$ and place it within $\Delta(asq)$. Note in particular, that $c'$ is on the circle of radius one centered at $p$ and that $c'$ is between $a$ and $s$. Also note $r'=s$, and $p'=q$. This implies $A(\Delta(asq))\ge A(\Delta(c'r'p'))=A(\Delta(crp))$. }\label{SecondYB}
\end{figure}
\begin{figure}[h]
\centering
         \includegraphics[width=.7\textwidth]{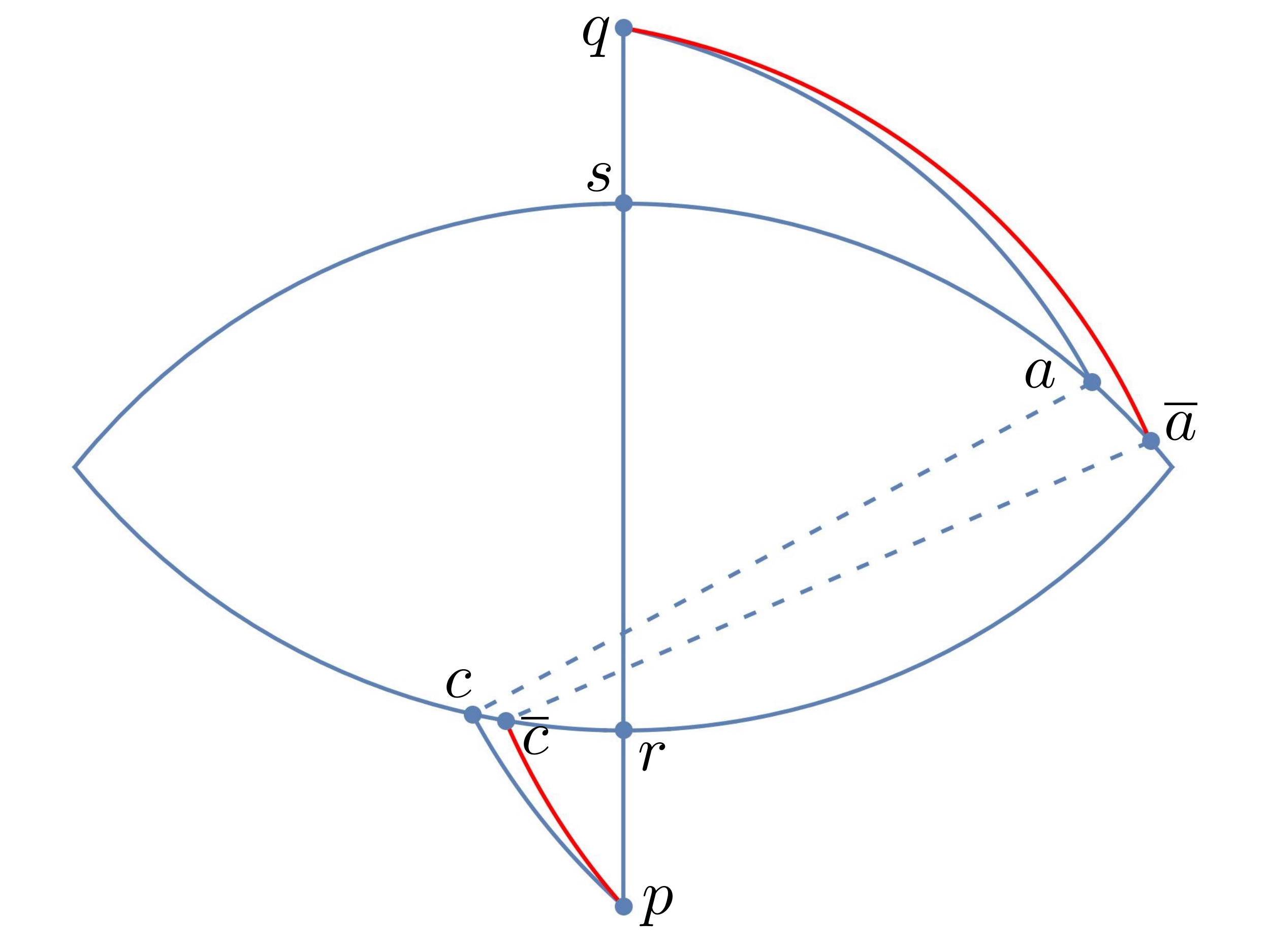}
                   \caption{This diagram helps explain why the area difference $A(\Delta(asq))-A(\Delta(crp))$ increases with the length difference  
                   $\overset{\frown}{ as}-\overset{\frown}{ cr}$. The key observation is that any pair $\overline a, \overline c$ with $\overset{\frown}{ \overline as}-\overset{\frown}{ \overline cr}\ge \overset{\frown}{  as}-\overset{\frown}{  cr}\ge 0$ must be arranged as in this diagram. In particular, $A(\Delta(\overline asq))\ge A(\Delta(asq))$ and  $A(\Delta(\overline crp))\le A(\Delta(crp))$, which implies the asserted monotonicity.}\label{ThirdYB}
\end{figure}
\par $(ii)$ Now suppose we have two other points $\overline a$ and $\overline c$ with $\overline a\in C(p)\cap D(q)$, $\overline c\in C(q)\cap D(p)$, and $|\overline a-\overline c|=1$. Also assume $\overset{\frown}{\overline as}- \overset{\frown}{\overline cr}\ge \overset{\frown}{ as}-\overset{\frown}{ cr}\ge 0$. This is the case provided that 
$$
a\in \overset{\frown}{\overline as}\quad \text{and}\quad \overline{c}\in \overset{\frown}{ cr}.
$$ 
See Figure \ref{ThirdYB}.   As we saw in part $(i)$, $A(\Delta(\overline asq))\ge A(\Delta(asq))$ and $A(\Delta(\overline crp))\le A(\Delta(crp))$. Therefore, 
$$
A(\Delta(\overline asq))-A(\Delta(\overline crp))\ge A(\Delta(asq))- A(\Delta(crp))\ge 0.
$$
\end{proof}
\begin{figure}[h]
\centering
         \includegraphics[width=.7\textwidth]{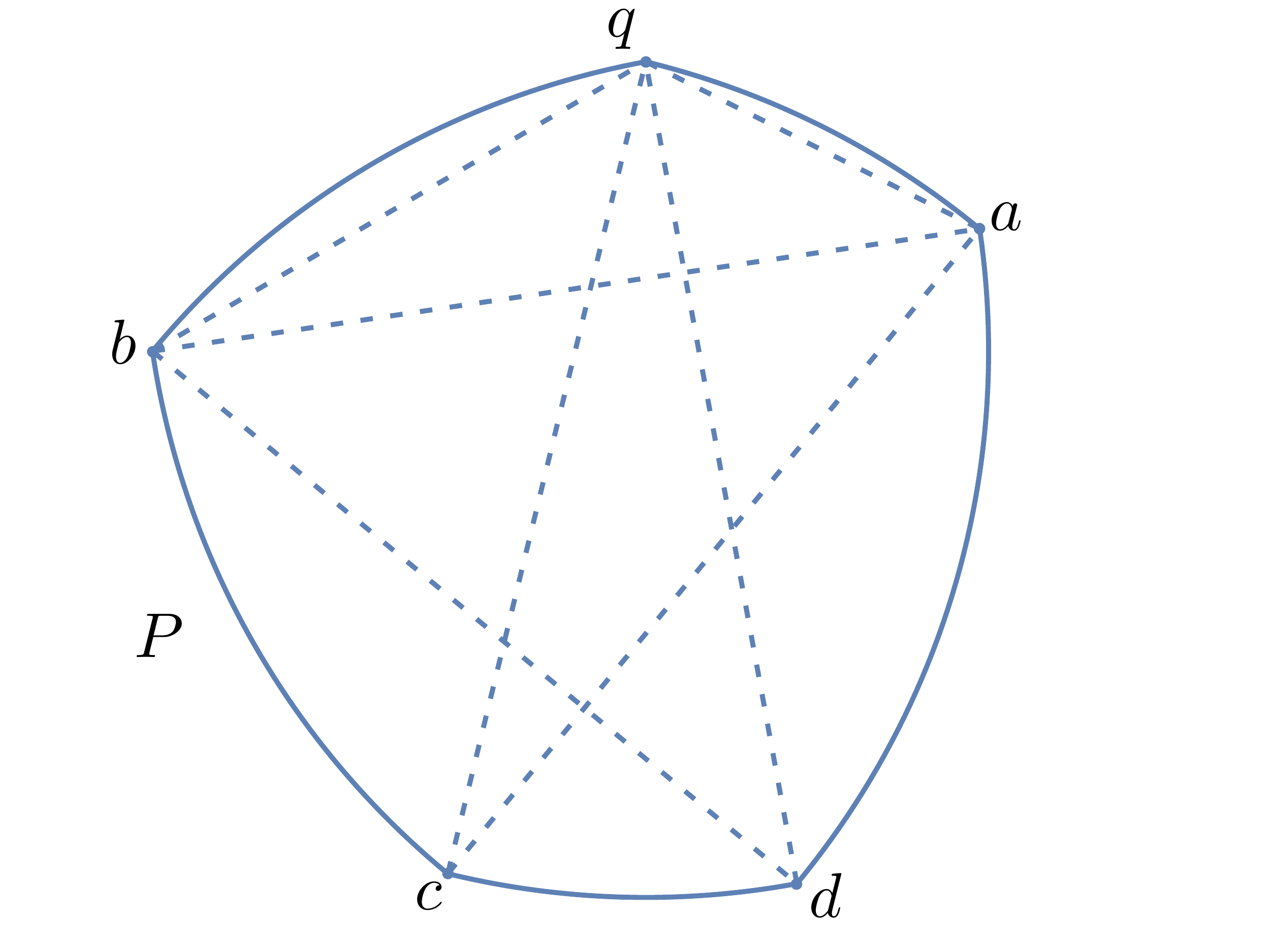}
                   \caption{A Reuleaux polygon $P$ as described in Lemma \ref{DistanceLemma}, which asserts that if $|c-d|\le|a-q|$ and $|c-d|\le |q-b|$, then $|c-d|\le |a-b|$.   }\label{FifthYB}
\end{figure}
\begin{lem}\label{DistanceLemma}
Suppose $P\subset \R^2$ is a Reuleaux polygon with neighboring vertices $c,d$.  Suppose $a,q$ are neighboring vertices of $P$ with 
$$
|a-c|=|q-c|=1
$$
and $q,b$ are neighboring vertices of $P$ with 
$$
|q-d|=|b-d|=1.
$$
If $|c-d|\le|a-q|$ and $|c-d|\le |q-b|$, then $|c-d|\le |a-b|$. 
\end{lem}
\begin{proof}
For $x,y,z\in \R^2$, we will write $\angle(xyz)$ the angle between $x-y$ and $z-y$.  Since $P$ has diameter one, $|a-q|\le 1$. If $|a-q|=1$, it can be verified that $c=b$, $d=a$, and $P$ is necessarily a Reuleaux triangle. The claim is trivial in this case. Otherwise, we assume $|a-q|<1$. It follows that 
$\angle(qca)< \pi/3$ and in turn that $\angle(cqa)> \pi/3$. Likewise, $\angle(dqb)> \pi/3$. As $\angle(cqd)\le \pi/3$, 
$$
\angle(aqb)=\angle(cqa)+\angle(dqb)-\angle(cqd)> \pi/3.
$$
It follows that the side between $a$ and $b$ of the (standard) triangle with vertices $a,q,b$ is not the smallest side.  See Figure \ref{FifthYB}. As
$|c-d|$ is not larger than the lengths of the other two sides of this triangle, it must be that $|c-d|\le |a-b|$.
\end{proof}

\begin{proof}[Proof of the Blaschke--Lebesgue theorem]
It suffices to show that a Reuleaux triangle $T$ has least area among all Reuleaux polygons. Indeed, suppose $K$ is a constant width curve and $K_\epsilon$ is a Reuleaux polygon with 
$$
A(K_\epsilon)\le A(K)+\epsilon. 
$$
Such a Reuleaux polygon exists by Proposition \ref{ApproxProp} and formula \eqref{IntByPartsArea}.  If $A(T)\le A(K_\epsilon)$, then 
$$
A(T)\le A(K)+\epsilon. 
$$
As this would hold for any $\epsilon>0$,  $A(T)\le A(K)$ and we would then conclude the Blaschke--Lebesgue theorem. 

\par In order to prove the claim, we will argue that that for any $N$-sided Reuleaux polygon $P$, with $N$ odd and $N>3$, there is another Reuleaux polygon $P'$ with $N-2$ sides and having smaller area than $P$.  In finitely many steps, we could then deduce that the Reuleaux triangle has area less than or equal to $P$.  To this end, choose a pair $c,d\in P$ of neighboring vertices for which the distance between $c$ and $d$ is as small as any other pair of neighboring vertices.  Let $q$ the vertex opposite $\overset{\frown}{cd}$. There are a pair of arcs $\overset{\frown}{bq}\subset C(d)$ and $\overset{\frown}{qa}\subset C(c)$ in the boundary of $P$.  There are also two solutions $z$ of  the equations 
$$
|b-z|=|a-z|=1. 
$$
We define $p=z$ to be the solution closer to the arc $\overset{\frown}{cd}$.   See Figure \ref{FourthYB}.
\begin{figure}[h]
\centering
         \includegraphics[width=.8\textwidth]{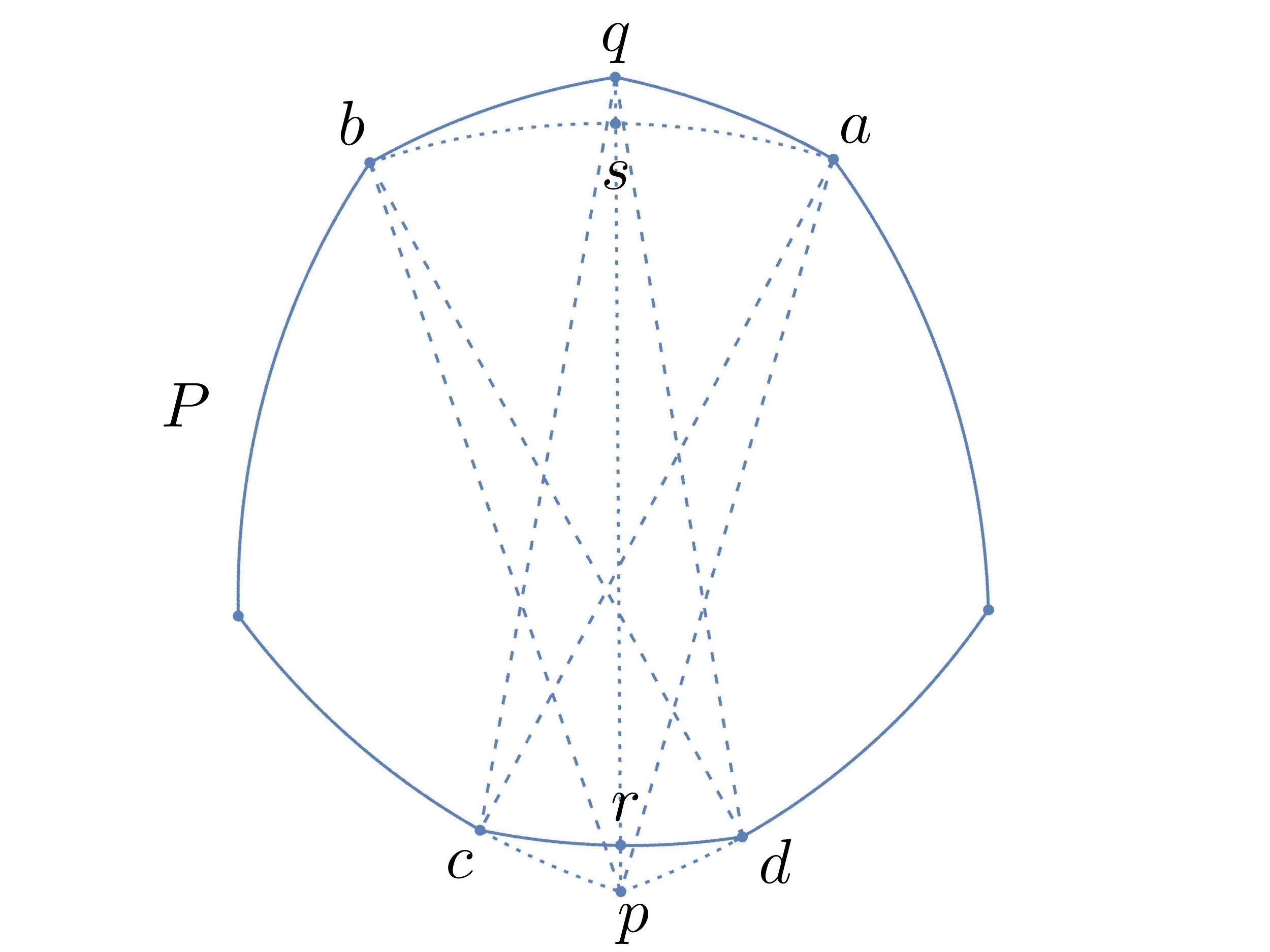}
                   \caption{A Reuleaux polygon $P$ as described in our proof of the Blaschke--Lebesgue theorem. We obtain an auxiliary Reuleaux polygon $P'$  from $P$ by replacing the arc $\overset{\frown}{cd}$ with the two arcs $\overset{\frown}{cp}$ and $\overset{\frown}{pd}$ and by replacing the two arcs $\overset{\frown}{bq}$ and $\overset{\frown}{qa}$ with $\overset{\frown}{ba}$. The key here is that $A(P)\ge A(P')$ and $P'$ has two fewer vertices than $P'$. }\label{FourthYB}
\end{figure}
\par We will construct a new Reuleaux polygon $P'$ from $P$ by replacing 
the arc $\overset{\frown}{cd}$ with the union of two arcs $\overset{\frown}{cp}$ and $\overset{\frown}{pd}$  and by replacing the two arcs 
$\overset{\frown}{bq}$ and $\overset{\frown}{qa}$ with the arc $\overset{\frown}{ba}\subset C(p)$.  See Figure \ref{FourthYB} for a detailed diagram.  It is routine to check that $P'$ has constant width.  Moreover, $p$ is a vertex of $P'$ while $c,d$ and $q$ are no longer vertices. In particular, $P'$ has $N-2$ vertices.  Furthermore, we claim that
\be\label{GeoLemmaClaim}
A(\Delta(asq))+A(\Delta(bsq))\ge A(\Delta(crp))+A(\Delta(drp)),
\ee
which would in turn imply $A(P)\ge A(P')$. Establishing this claim would complete our proof. 

\par By Lemma \ref{DistanceLemma}, $|c-d|\le |a-b|$. In particular, $\overset{\frown}{ab}\ge \overset{\frown}{cd}$.   That is, 
\be\label{TotalLengthIneq}
\overset{\frown}{as}+\overset{\frown}{sb}\ge \overset{\frown}{cr}+ \overset{\frown}{rd}
\ee
for points $s$ and $r$ which lie on the line segment between $p$ and $q$ with $|s-p|=|r-q|=1$. If 
\be\label{TwoGoodDirectionIneq}
\overset{\frown}{as}\ge \overset{\frown}{cr}\quad \text{and}\quad \overset{\frown}{sb}\ge \overset{\frown}{rd},
\ee
then Lemma  \ref{GeometricLemma} $(i)$ implies 
$$
A(\Delta(asq))\ge A(\Delta(crp))\quad \text{and}\quad A(\Delta(bsq))\ge A(\Delta(drp)).
$$
We conclude \eqref{GeoLemmaClaim} after adding these inequalities.

\par Alternatively, let us suppose that one of the inequalities \eqref{TwoGoodDirectionIneq} goes the other way. For example, let's assume 
$$
\overset{\frown}{as}\ge \overset{\frown}{cr}\quad \text{and}\quad \overset{\frown}{sb}\le \overset{\frown}{rd}.
$$
In view of \eqref{TotalLengthIneq}, we have 
$$
\overset{\frown}{as}-\overset{\frown}{cr}\ge \overset{\frown}{rd}-\overset{\frown}{sb}\ge 0.
$$
By part $(ii)$ of Lemma \ref{GeometricLemma},
$$
A(\Delta(asq))-A(\Delta(crp))\ge A(\Delta(drp))-A(\Delta(bsq)).
$$ 
This inequality is equivalent to \eqref{GeoLemmaClaim}. As a result, we have verified inequality  \eqref{GeoLemmaClaim}
in all cases and have therefore proved the Blaschke--Lebesgue theorem. 
\end{proof}
\appendix 
\bibliography{BLbibRev}{}

\bibliographystyle{plain}

\typeout{get arXiv to do 4 passes: Label(s) may have changed. Rerun}

\end{document}